\documentclass[reqno]{amsart}
\usepackage{amssymb,latexsym,amscd} 
\usepackage{enumerate}[1.)]
\usepackage{pb-diagram}  
\renewcommand{\eqref}[1]{(\ref{#1})}   
\theoremstyle{plain}
\newtheorem{theorem}{Theorem}[section]
\newtheorem{theorema}{Theorem}

\newtheorem{corollary}{Corollary}[section]

\newtheorem{lemma}{Lemma}[section]

\newtheorem{proposition}{Proposition}[section]
\newtheorem{algorithm}{Algorithm}

\newtheorem{definition}{Definition}

\theoremstyle{remark}

\newcommand {\C}{{\mathbb{C}}}
\newcommand {\R}{{\mathbb{R}}}
\newcommand {\Z}{{\mathbb{Z}}}
\newcommand {\N}{{\mathbb{N}}}

\newcommand {\LL}{{\mathcal{L}}}

\newcommand {\eps}{{\varepsilon}}

\newcommand{\diff}{\textnormal{ d}}

\begin{document}

\title[Strong orthogonality and cusp forms]{Strong orthogonality between the M\" obius function, additive characters, and Fourier coefficients of cusp forms}
\date{\today}
\author{\'Etienne Fouvry}
\address{ Universit\' e Paris Sud, Laboratoire  de  Math\' ematique, UMR 8628, Orsay, F--91405 France, CNRS, Orsay, F--91405, France}
\email{Etienne.Fouvry@math.u-psud.fr}
\author{Satadal Ganguly}
\address{Theoretical Statistics and Mathematics Unit, Indian Statistical Institute, 203 Barrackpore Trunk Road, 
Kolkata 700108, India}
\email{sgisical@gmail.com}
 

\begin{abstract}  Let  $\nu_{f}(n)$ be  the $n$-th nomalized Fourier coefficient of a Hecke--Maass cusp form  $f$ for 
${\rm SL }(2,\Z)$ and let $\alpha$  be a real number. 
We prove strong oscillations of the argument of $\nu_{f}(n)\mu (n)
\exp (2\pi i n \alpha )$ as $n$ takes consecutive integral values.  \end{abstract}

\maketitle

\renewcommand{\theenumi}{(\roman{enumi})}


\section{Introduction}

Fourier coefficients of cusp forms  are mysterious objects and an interesting question is how,
for a fixed form, its Fourier coefficients are distributed. There are many results from which the distribution appears to be highly random. 
For example, consider
 the following uniform bound on linear 
forms involving normalized Fourier coefficients $\nu_f (n)$ of a Maass cusp form $f$ (see \S 2 for the normalization) 
twisted by an additive character $e(\alpha):=\exp(2\pi i \alpha)$ (see \cite[Theorem 8.1]{Iw2}):
\begin{equation}\label{sum}
\sum_{|n|\leq N} \nu_f(n)e(n\alpha) \ll_f N^{1/2}\log 2N.
\end{equation}
We emphasize that the implied constant here depends only on $f$ and not on the real number $\alpha$. 
The estimate \eqref{sum} signifies an enormous amount (square-root of the length of summation) of cancellations. 
This means  that the Fourier coefficients are quite far from being aligned with the values of any fixed additive character and
therefore, the bound \eqref{sum} can be interpreted as manifestation of \textit{non-correlation} or a kind of ``\textit{orthogonality}'' between 
the Fourier coefficients of $(\nu_f(n))$ and the sequence $(e(n\alpha))$. Following \cite{Sarnak-Lecture} and \cite{GT}, we say
 two sequences  
$(x_n)$ and $(y_n)$ of complex numbers are \textit{asymptotically orthogonal} (in short, ``orthogonal") if 
\begin{equation}\label{ortho}
\sum_{1 \leq n \leq N} x_n y_n = o\Bigl( \bigl(\sum_{n \leq N} |x_n|^2\bigr)^{\frac{1}{2}} 
\bigl(\sum_{n \leq N} |y_n|^2\bigr)^{\frac{1}{2}}\Bigr)
\end{equation}
as $N \longrightarrow \infty$; and \textit{strongly asymptotically orthogonal} (in short, ``strongly orthogonal") if 
\begin{equation}\label{sortho}
\sum_{1 \leq n \leq N} x_n y_n = O_A\Bigl( (\log N)^{-A} \sum_{n \leq N} |x_n y_n|\Bigr)
\end{equation}
for every $A \geq 0$, uniformly for $N \geq 2$. 
The bound \eqref{sum} shows that the two sequences $(\nu_f(n))$ and $(e(n\alpha))$ are strongly orthogonal.
The question we seek to answer is whether strong orthogonality is manifested if, instead of the sum in \eqref{sum}, we  consider the corresponding sum 
over primes: 
\begin{equation}\label{primesum}
 \mathcal{P}_f(X, \alpha):=\sum \limits_{\substack{p \leq X\\ p \textnormal{ prime}}} \nu_f(p)e(p\alpha).
\end{equation}
 Another interesting question is to ask whether  the sequences $(\nu_f(n)e(n\alpha))$ and $(\mu(n))$ are strongly orthogonal.
The {\it M\"obius Randomness Law} (see \cite[\S 13.1]{IK})
asserts that the sequence $(\mu(n))$ should be orthogonal to any ``{\it reasonable}\rq\rq{} sequence.  Sarnak has recently posed a more precise
conjecture in this direction and we refer  the reader to \cite{Sarnak-Lecture}, \cite{BSZ}, \cite{CS} \& \cite{SU} for 
 recent developments on this theme.\\ This question leads us to investigate cancellations in the sum \textit{dual} to \eqref{sum} (in the sense of Dirichlet convolution):
\begin{equation}\label{musum}
 \mathcal{M}_f(X, \alpha):=\sum \limits_{1 \leq n \leq X}\mu(n) \nu_f(n)e(n\alpha).
\end{equation}
Using classical techniques from analytic number theory and a recent impressive result due to Miller \cite{Mi}, we establish bounds for both \eqref{primesum} and \eqref{musum} that go beyond strong orthogonality, at least when
 $f$ is a Maass cusp form for the full modular group ${\rm SL}(2, \mathbb{Z})$ (of arbitrary weight 
and Laplace eigenvalue).   
  Here our definition of Maass form is general enough to include 
holomorphic modular forms.
Our main theorem is:
\begin{theorem}\label{mainthm1}  There exists an effective absolute $c_{0}>0$ such that, for any 
Maass cusp form  $f$ for the group ${\rm SL}(2, \mathbb{Z})$,  of arbitrary weight and Laplace eigenvalue, 
there exists an effective constant $C_0(f)>0$
   such that one has the inequalities
\begin{equation}\label{139}
 \Bigl| \, \mathcal{P}_f(X, \alpha))\, \Bigr| \leq C_0(f)  X\exp \bigl (-c_0 \sqrt {\log X}\bigr),
 \end{equation}
 and 
 \begin{equation}\label{144}
 \Bigl\vert \, \mathcal{M}_f(X, \alpha)\, \Bigr\vert \leq C_0(f)  X\exp \bigl (-c_0 \sqrt {\log X}\bigr),
\end{equation}
 for every  $\alpha \in \R$ and $X \geq 2$.
\end{theorem}

The strong orthogonality we mentioned above now follows from the lower bound given in Proposition \ref{strongortho}.  
In particular, \eqref{144} says that the M\" obius Randomness Law is true in the
 case of the function $n\mapsto \nu_{f}(n)e (n\alpha)$ in a strong sense. 
Theorem \ref{mainthm1} can also be interpreted of as the Prime Number Theorem (denoted henceforth by PNT) for Fourier coefficients of
cusp forms with additive twists. In fact, \eqref{musum} is the ${\rm GL}(2)$ analogue of a result of Davenport (see \cite{Dav} or  \cite[\S 13.5]{IK}) 
which says that for any real number $\alpha$, $X\geq 2$ and $A >0$, we have
the bound
\begin{equation}\label{Dav}
\sum_{n\leq X}\mu(n)e(n\alpha) \ll_A X(\log X)^{-A}.
\end{equation} 
The weaker bound here is a reflection of the exceptional zero (see \cite[Chap. 5]{IK}) which is not yet ruled out in the ${\rm GL}(1)$ situation. 
By contrast,  Hoffstein and 
Ramakrishnan \cite{HR} have shown that there are no exceptional zero for $L$-functions on ${\rm GL} (2)$  
that are not associated to grossencharacters of quadratic fields.

As soon as $\alpha$ has a sufficiently good approximation by rationals, for example, if we have suitable 
control over the  infinite continued fraction expansion of $\alpha$, then the upper bound (6) is highly improved and we obtain a power saving. 
The most
typical case is  the \textit{golden ratio}
$
\alpha= \rho = (1+\sqrt 5)/2.
$ In that particular case, we know that for every
$X>2$, there is a fraction $a/q$, $(a, q)=1$, satisfying \eqref{Dir} and the
inequality
$\sqrt X < q < 2 \sqrt X$. The formula \eqref{final} then directly leads to the following corollary

\begin{corollary}\label{Golden}
We have the bound
$$
\mathcal{M}_f (X,\rho)\ll X^{\frac{59}{60}+\eps}.
$$
\end{corollary}

Theorem \ref{mainthm1} is suitable for invoking the circle method. For instance, reserving the
letter $p$ to denote primes, we have the following corollary. The proof follows directly  from the basic identity of the circle method and the Parseval formula.

\begin{corollary}\label{CIRC} There exists an effective  absolute $c_{0}>0$, such that  for any Maass cusp form $f$ 
for the group ${\rm SL}(2, \mathbb{Z})$ there exists an effective constant $C_{0}(f)$ such
 that
    one has the inequality
\begin{equation}\label{circ1}
\Bigl\vert \,\underset{N=p+a+b}{\sum \ \sum\  \sum }\nu_f (p)\,\alpha_a \,\beta_b\, \Bigr\vert \leq C_0(f) N \exp\bigl( -c_0 \sqrt{ \log N}\bigr)
||\alpha_{N}||\, ||\beta_{N}||,
\end{equation}
for every $N\geq 4$, for every sequence of complex numbers $(\alpha_{a})_{a\geq 1}$ and $(\beta_{b})_{b\geq 1}$
where we denote $||\alpha_{N}||^2=\sum\limits_{1\leq a \leq N} \vert \alpha_a \vert^2$ and
$||\beta_{N}||^2=\sum\limits_{1\leq b \leq N} \vert \beta_b \vert^2$. In particular, for the Ramanujan $\tau$--function and for $N\geq 6$,  one has the inequality
\begin{equation}\label{circ2}
\Bigl\vert \,\underset{N=p_1+p_2+p_3}{\sum \ \sum\  \sum }\tau (p_1) \Bigr\vert 
\leq C_0 N^{15/2} \exp\bigl( -c_0 \sqrt{ \log N}\bigr),
\end{equation}
where $C_0$ and $c_0$ are some positive constants, both effectively computable.
 \end{corollary}

\noindent
To see the interest of \eqref{circ1}, suppose that the sequences $(\alpha_a)$ and $(\beta_b)$ are the 
characteristic functions of sequences of positive integers $\mathcal A$ and $\mathcal B$, with counting functions 
$A (N)$ and $B( N)$, up to $N$. If $f$ is holomorphic, Deligne's bound \eqref{Ramu} implies the trivial bound
  $$
\Bigl\vert \,\underset{N=p+a+b}{\sum \ \sum\  \sum }\nu_f (p)\,\alpha_a \,\beta_b\, \Bigr\vert �\ll A(N)   B(N).
$$
Hence, \eqref{circ1} is interesting as soon as the sequences $\mathcal A$ and $\mathcal B$ are dense enough, which means the condition
$
A(N) B(N) \gg N^2 \exp\bigl( -2c_0 \sqrt {\log N}\bigr),
$
is satisfied for sufficiently large $N$; for instance, when $\mathcal A$ and $\mathcal B$ are the sequence of primes or
certain sequences of smooth numbers:
$
\mathcal{A}=\mathcal{ B}=\bigl\{ n\ :\ p\mid n \Rightarrow p< \exp\bigl( \log^\theta n)\bigr\},
$
where $\theta$ is any fixed real number satisfying $\theta >1/2.$ 
Note that \eqref{circ2} is trivial if $N$ is even; but if $N\geq 7$ is odd, the famous Vinogradov's Theorem gives the lower bound
$$
\underset{N=p_1+p_2+p_3}{\sum \ \sum\  \sum }1
\gg N^2 (\log N)^{-3}.
$$
In other words, \eqref{circ2} shows a lot of oscillations of the coefficient $\tau (p_1)$ in the expression of $N$ of the form $N=p_1+p_2+p_3$. 
The same is true for the coefficient $\tau (p_1) \tau (p_2) \tau (p_3)$.

Our proof  is along the lines of Davenport's \cite{Dav} and it follows different paths depending on the 
diophantine nature of $\alpha$: whether or not it is near a rational number with denominator sufficiently small. 
In the first case; i.e., when $\alpha$ belongs to the so called {\it major arcs}, we can use a suitable PNT for automorphic 
$L$-functions.\\
The formulas (6) and (7), though apparently not equivalent, are recognized
to have the same depth. We only prove the bound (7) since the proof  of  (7) is more delicate than the proof of  (6). One reason for this
is that we need to prove the required PNT Theorem \ref{pnt2} from scratch.
 
For {\it  minor arcs}, i.e., when $\alpha$ cannot be approximated by rationals with small denominators, we apply Vinogradov\rq{}s method for 
exponential sum via Vaughan\rq{}s identity. 
Thus we are led to the so called sums of type I and type II. In estimating the type II sum, the more difficult one, we encounter a sum which is 
naturally related to the symmetric square lift of the Maass form $f$. 
A result of Miller (see \cite[Theorem 1.1]{Mi}) suitably adapted to our requirement (see Lemma \ref{Millgen}) is crucial here. 
Miller\rq{}s theorem, which is a consequence of
Voronoi summation formula for ${\rm GL}(3)$ (see \cite{MS1} and also \cite{GL1}), says the following:
For a cusp form on ${\rm GL}(3,\Z)\backslash {\rm GL}(3,\R)$ with Fourier coefficients $a_{r,n}$, one has 
\begin{equation}\label{Mill100}
\sum_{n\leq T} a_{r,n}  e(n\alpha) \ll T^{\frac{3}{4}+\eps},
\end{equation}
 where the implied constant depends only the form, the integer $r$ and $\varepsilon$. This is why we confine ourselves
to the level one situation as the analogous result in the case of a general level, though expected, is not yet available.

However, in certain ranges of the variables \eqref{Mill100}  gives  trivial bounds and we need to appeal 
to the oscillations of the additive character $n\mapsto e(\alpha n)$. Here the condition that $\alpha$ belongs to the 
minor arcs becomes important (see  the classical Lemma \ref{IwKowp.346} below).\\
This brings us to another difference between the proofs of (6) and (7).  This is due to the difference between 
 the combinatorial structures of $\Lambda$ and $\mu$.
It is more  difficult in this context to apply the Vaughan identity (89) for the M\"obius function
 than its classical analogue for the von Mangoldt function. The reason is that one needs to
control the greatest common divisors of the variables of summations in the case of the M\"obius function 
whereas this problem disappears completely in the case of  the von Mangoldt function (as two distinct primes are coprime). This problem is amplified by the fact that 
$n\mapsto \lambda_f (n)$ is not completly 
multiplicative (see Lemma \ref{Hec1}). To circumvent this, we introduce a function $\lambda^*$ (see \eqref{defto}) to average out the 
chaotic behavior of the function $\lambda_f$ (see \eqref{defto}). Then  the average behaviour of the function $\lambda^*$ is controlled thanks
to the recent result of Lau and L\"{u} \cite{Lau-Lu} 
on higher moments of Fourier coefficients of Maass cusp forms.
In the case where $f$ is holomorphic, the proof is highly shortened due to Deligne's bound.

\subsection{Some remarks}\label{somer}

\noindent
\textit{Remark 1.}
We expect both the sums \eqref{primesum} and \eqref{musum} to be quite small, at least on average. Indeed, it is relatively easy to see that \textit{square-root cancellations} take place in both the sums 
 in the mean-square sense. By the Parseval formula and the Rankin-Selberg estimate 
(see \eqref{L2}) it readily follows that
\[
\int_0^1 |\mathcal{M}_f(X, \alpha)|^2 \diff \alpha \leq \sum_{1 \leq n \leq X}|\nu_f(n)|^2 \ll_f X,
\]
and similarly for $\mathcal{P}_f(X, \alpha)$.
Using a simple observation of Oesterl\'e (see \cite[\S 1]{Murty-Sankar}) we can even get the pointwise bound
$$
\mathcal{M}_f(X, \alpha), \ \mathcal{P}_f(X, \alpha) \ll_{\alpha, \epsilon, f} X^{\frac{1}{2}+\eps}
$$
 for any $\eps>0$, for almost all $\alpha$ (in the sense of Lebesgue measure). Recall the famous theorem of Carleson \cite{Carl} which says that if $(c_n)$ is a sequence of complex
numbers satisfying
$
\sum_{n=1}^\infty |c_n|^2 < \infty,
$
then the Fourier series $\sum_{n=1}^\infty c_n e(n \alpha)$ converges for almost all real $\alpha$. Now the Rankin-Selberg
estimate \eqref{L2} and partial summation allows us to apply the theorem to the sequence
$
c_n =\frac{\nu_f(n)}{n^{1/2+\eps}},
$
where $\eps>0$ is arbitrary and draw the desired conclusion.
 Of course, this line of arguments does not give any non-trivial bound for any specific value of $\alpha$.

\noindent
\textit{Remark 2}. 
Regarding the sum appearing in \eqref{musum},
It turns out  that proving mere orthogonality between $(\mu(n))$ and the sequence $(\nu_f(n)e(n\alpha))$ is not very difficult. 
 Indeed, bounds of the type
$$
\sum_{1\leq n\leq X}\  \bigl\vert\,\lambda_f(n) \, \bigr\vert \ll_f  X (\log X)^{-\delta}
$$
for some $0<\delta \leq1$ for normalized Hecke eigenvalues $\lambda_f(n)$ of holomorphic forms $f$ have been known
for quite some time. See, for example,  \cite{EMS}, \cite{Ram1}, and \cite{Ran2}. For Maass forms also, one can easily
 conclude that 
\[
\sum_{1\leq n\leq X}\  \bigl\vert\,\lambda_f(n) \, \bigr\vert =o(X)
\]
as $X \longrightarrow \infty$ from  \cite[eqn. (66)]{Roman} and  \cite[Theorem 2]{Elliott}.
Orthogonality follows from this bound and \eqref{L2}. 
 However, as the lower bound \eqref{lowerbd} shows, it is not possible to save an arbitrary large power of logarithm
in the above sum. The situation is exactly similar for the sum over primes.

\subsection{Notation and convention} We follow the well known notations and conventions described below:

\noindent $\bullet$  $d(n)$ denotes the number of divisors of the integer  $n$, $d_{3}(n)$ is the number of ways of writing
$n=n_{1}n_{2}n_{3}$, where the $n_{i}$ are positive  integers. The number of prime divisors of $n$
is $\omega (n)$ and $\varphi(n)$ denotes the number of moduli coprime to $n$.

\noindent $\bullet$ $(m,n)$ and $[m, n]$ denote the $g.c.d$ and the $l.c.m.$ of integers $m$ and $n$.

\noindent $\bullet$ $\eps$ denotes a positive unspecified real number, different in different occurences.

\noindent $\bullet$ In asymptotic formulae of the form  $A(X) =B(X)+O_{\beta}(C(X)) \textnormal{ or }\\
 A(X) \ll_{\beta} B(X)$
the suffix $\beta$ signifies the dependence of the implied constant on some parameter $\beta$ which is fixed with respect to the variable $X$. However, dependence of various parameters will sometimes be suppressed when it is either not important for our purpose or is clear from the context.\\
\noindent $\bullet$ $w \sim W$ denotes $W<w \leq 2W$.

\vspace{3mm}


\textit{Acknowledgement:}  During the preparation of this work, E.F. benefited from the support
of Institut Universitaire de France.
S.G. would like to thank project ARCUS and the Laboratoire de Math\' ematique of the Universit\' e Paris Sud for arranging his visit
 during which this work was started. The authors thank F.~Brumley, D.~Bump, R.~Holowinsky, D.~Goswami, 
E.~Kowalski, Y.K.~Lau,  Ph.~Michel, S.D.~Miller, C.S.~Rajan, M.~Ram Murty, O.~Ramar\'e, E.~Royer,
P.~Sarnak, and J. Wu for many helpful remarks.


\section{Background on Maass forms}

\subsection{Maass forms}
This section contains a very brief account of the theory of Maass forms based primarily on   \cite[\S 4,\,5, and 6]{DFI}. See
also  \cite[\S 2.1]{Bump}. One of our aims  is to explain the embedding of the holomorphic modular forms in the space of Maass forms so that we can give a unified proof of our result. Although we shall work only with forms of 
level one, we consider a general level $q$ in this section. 

Let $k$ be an integer, $q$ a positive integer, and $\chi$, a Dirichlet character modulo $q$ that satisfies the consistency condition
$
  \chi ( - 1 ) = ( - 1 )^k.
$
Such a character gives rise to a character of the Hecke congruence group $\Gamma_0 ( q )$ by declaring $\chi (\gamma ) = \chi ( d )$ for \\
$\gamma = \left(\begin{array}{cc}
  a & b\\
  c & d
\end{array}\right) \in \Gamma_0(q).$
For $z\in \mathbb{H}$, the upper half plane, we set \[j_{\gamma} ( z ) := (cz + d)| cz + d|^{-1} = e^{i \textnormal{ arg} ( cz + d )}.\]
A function $f :\mathbb{H} \longrightarrow \mathbb{C}$ that satisfies the
  condition \[f ( \gamma z ) = \chi ( \gamma ) j_{\gamma} ( z )^k f ( z )\]
  for all $\gamma \in \Gamma_0(q)$ is called is called an {\em automorphic function} of weight $k$, level $q$,
  and character (also called {\em nebentypus}) $\chi$. The Laplace operator of weight $k$ is defined by
\[
   \Delta_k =  y^2 \left( \frac{\partial^2}{\partial x^2} +
  \frac{\partial^2}{\partial y^2} \right) - iky \frac{\partial}{\partial x},
\]
 and a smooth automorphic function $f$ as above that is also an eigenfunction of the Laplace operator; i.e.,
$\left(\Delta_k + \lambda \right)f = 0$
for some complex number $\lambda$, is called a {\em Maass form} of crresponding weight, level, character, and 
Laplace eigenvalue $\lambda$. One writes $\lambda(s)=s(1-s)$ and $s=1/2 +ir$, with $r, s \in \mathbb{C}$,  $r$ being known as
 the \textit{spectral parameter}. 
It is related to the Laplace eigenvalue $\lambda$ by the equation  
\begin{equation} \label{1/4+r2}
 \lambda =\frac{1}{4}+r^2.
 \end{equation}
Beware that some authors define `Maass forms' to be what are Maass forms of weight zero in our setting.
 One can show that $\lambda(|k|/2)$ is the lowest eigenvalue of $-\Delta_k$ and
if $k \geq 0$ (resp. $k \leq 0$) and $f$ is a Maass form with this lowest eigenvalue, then the Cauchy-Riemann equation shows that
$y^{-k/2}f(z)$ (resp. $y^{k/2} \overline{f(z)}$) is a holomorphic function.  These holomorphic functions are actually the classical \textit{modular forms} (see \cite[\S 4]{DFI}). 
 A fact that we require is that the Laplace eigenvalue $\lambda(s) =s(1-s)$ of a Maass cusp form which is not induced from a 
holomorphic form must satisfy 
(see \cite[cor.~4.4]{DFI})
\begin{equation}\label{Karzai}
\Re s = \frac{1}{2} \ \textnormal{ or } \ 0 <s <1.
\end{equation}
However, the \textit{Selberg eigenvalue conjecture} asserts that the latter case never occurs (see \S 3.2 also).

\subsection{Normalizations of Fourier coefficients}
Given a holomorphic cusp form $F$ with 
a Fourier expansion at the cusp at $\infty$ of the form
$$
F(z) = \sum_{n \geq 1} a_F(n) e(n z),
$$
we define the \textit{normalized Fourier coefficients} of a holomorphic cusp form $F$ to be 
\begin{equation}\label{psi}
\psi_F (n) = a_F(n)/n^{(k-1)/2},
\end{equation}
where $k$ is the weight of $F$. 

Now we come to Maass forms.
We consider Maass cusp forms only. See \cite[\S 4]{DFI} for the definition. We shall denote the space of Maass forms of level $q$, weight $k$, and
character $\chi(\textnormal{mod }q)$ by $\mathcal{C}_k(q, \chi)$. A form in this space admits Fourier expansion at the cusp at $\infty$ in 
terms of Whittaker functions 
$W_{\alpha, \beta}$ as follows (see \cite[eqn. (5.1)]{DFI}):
 \[
 f(z) = \sum_{n \neq 0}\rho_f(n)W_{\frac{k n}{2|n|}, ir} (4\pi |n|y)e(nx),
 \]
 where $r$ is the spectral parameter.  
When we speak of Maass cusp forms, we shall always assume that they have norm one; i.e., 
$ \left  \langle f, f \right\rangle =1$ (see \cite[eqn. (4.37)]{DFI}). We define the \textit{normalized Fourier coefficients} of a Maass cusp 
form $f$ (see \cite[Chap. 8]{Iw2}) by
 \begin{align}\label{definorm}
 \nu_f(n) := \left( \frac{4 \pi |n|}{\cosh \pi r }\right)^\frac{1}{2} \rho_f(n)
 \end{align}
provided $f$ is not induced from a holomorphic form; i.e., the Laplace eigenvalue of $f$ is not $\lambda(|k|/2)$. Note that if $f$ is such a Maass cusp form, 
then by \eqref{Karzai}, the spectral parameter $r$ satisfies $
r \in \mathbb{R} \ \textnormal{ or } \ 0 <\frac{1}{2} + i r <1,$
and therefore, $$
{\pi}^{-1} \cosh \pi r =\Gamma(1/2 +i r)^{-1} \Gamma(1/2 -i r)^{-1} \neq 0.
$$ Now we consider Maass cusps forms which are induced from the holomorphic modular forms.
Let $F$ be a holomorphic form of weight $k \geq 0$.
The Fourier coefficients of $F$ are related to the coefficents $\rho_f (n)$ where $f$ is the 
Maass cusp form associated to $F$ in the following way:
\[
f(z) = y^{k/2}F(z) \ \textnormal{ or } \ f(z)=y^{k/2}\overline{F}(z) .
\]
In the first case, the weight of the induced Maass form is $k$ and in the second, it is $-k$.
We know that in both cases the Laplace eigenvalue is $\lambda(k/2)$ and thus the spectral parameter is given by 
$
r= -i \frac{k-1}{2}. 
$
Now the Whittaker function has the property (see \cite[eqn. (4.21)]{DFI}) that 
$
W_{\alpha, \alpha -1/2} (y) = y^{\alpha} {\rm e}^{-y/2}.
$
Using this fact, we infer that (see \eqref{psi}) for $f(z) = y^{k/2}F(z)$,
$
\rho_f (n) = \frac{a_F (n)}{(4 \pi n)^{k/2}} = \frac{\psi_F(n)}{n^{1/2}(4 \pi)^{k/2}}
$
for $n \geq 1$, and $\rho_f (n) =0$ for $n \leq 0$. Similarly, when $f(z)=y^{k/2}\overline{F}(z)$, we have 
$
\rho_f (n) = \frac{\overline{a_F (n)}}{(4 \pi n)^{k/2}} = \frac{\overline{\psi_F(n)}}{n^{1/2}(4 \pi)^{k/2}}
$
for $n \geq 1$, and $\rho_f (n) =0$ for $n \leq 0$.
Accordingly,  for $f(z) =  y^{k/2}F(z)$ (resp. $f(z)=y^{k/2}\overline{F}(z)$) where $F$ is a holomorphic cusp form,
we define $
\nu_f(n)=\frac{\psi_F (n)}{(4 \pi)^{(k-1)/2}} \ \ \Bigl(\textnormal{resp.} \ \frac{\overline{\psi_F (n)}}{(4 \pi)^{(k-1)/2}}\Bigr)
$
for $n \geq 1$ and $\nu_f(n)=0$ otherwise. 

\subsection{Hecke operators }

The  definition of the $n$-th Hecke operator $T_{n, \chi} $, $n \geq 1$ acting  on the space of 
modular forms of level $q$, weight $k$, and character $\chi(\textnormal{mod }q)$ is given  by
\begin{equation*}
T_{n,\chi}\,: F(z)\mapsto   ( T_{n, \chi} F )( z) = \frac{1}{n} \sum_{ad = n} \chi(a)a^k \sum_{b ( \textnormal{mod } d )} F \left( \frac{az + b}{d} \right).
    \end{equation*}
For an eigenfunction $F$ of $T_n$, we shall denote the eigenvalue by $\lambda_F (n)$. 
If $F$ is a \textit{primitive form} (i.e., newform) then its Fourier coefficients $a_F (n)$ are related to the eigenvalues $\lambda_F (n)$ by 
\begin{equation}\label{basicr}
a_F (n)=a_F (1) \lambda_F (n),
\end{equation}
and, moreover,  $a_F (1) \neq 0$. 
Hence, the Fourier coefficients and the Hecke eigenvalues coincide up to a multiplicative factor that depends only on the form $F$. 
We define the action of the $n$-th Hecke operator  $T'_{n,\chi}$  on $\mathcal{C}_k(q, \chi)$ by (see \cite[Chap. 6]{DFI})
\begin{equation*}
T'_{n,\chi}\,: f(z)\mapsto   ( T'_{n, \chi} f )( z) =  \frac{1}{\sqrt{n}} \sum_{ad = n} \chi(a) \sum_{b ( \textnormal{mod } d )} f \left( \frac{az + b}{d} \right).
    \end{equation*}
Note that this definition is independent of the weight $k$. The Hecke theory for Maass forms is parallel to the theory for modular forms and an important fact is that there is an orthonormal basis (called Hecke basis) of Maass cusp forms consisting of forms that are common eigenfunctions of the 
Hecke operators $T'_{n, \chi}$ with $(n, q)=1$. 
 The forms in a Hecke basis will be called 
 \textit{Hecke-Maass cusp forms}.  A Hecke-Maass cusp form in the new subspace (consisting of forms that are not linear combination of forms induced from lower levels) is called a \textit{newform} or a \textit{primitive form}. 
Note that a Hecke-Maass cusp form of level one is trivially a
primitive form. 
The Hecke eigenvalue $\lambda_f(n)$ and the normalized Fourier coefficient $\nu_f(n)$ of a Hecke-Maass cusp form are related by 
\begin{equation}\label{basicrelation}
\nu_f(\pm n) = \nu_f(\pm 1) \lambda_f(n); n\geq 1.
\end{equation}
Moreover,  for a Maass cusp form $f$ which is not induced from a holomorphic form, we have the relation $\nu_f(-1) = \eps_f \nu_f(1),$
where $\eps_f= 1$ or $-1$ and the form $f$ is accordingly called \textit{even} or \textit{odd}. The following proposition is easy to check.
\begin{proposition}\label{relationhecke} 
Suppose $F$ is a cusp form of weight $k$, level $q$ and character $\chi (\textnormal{mod }q)$ and let $f(z) = y^{k/2}F(z)$ (resp. $f(z) = y^{k/2}\overline{F}(z)$) be 
the associated Maass cusp form in $\mathcal{C}_k (q, \chi)$ (resp.$\mathcal{C}_{-k} (q, \chi)$) with Laplace 
eigenvalue $\lambda(k/2)$. Then $F$ is an eigenfunction of the $n$-th Hecke operator if and only if $f$ is. Moreover, the $n$-th Hecke eigenvalues $\lambda_F (n)$ and $\lambda_f (n)$ of $F$ and $f$ respectively are related by
\[
\lambda_f (n) = \frac{ \lambda_F (n)}{n^{(k-1)/2}} \ \Bigl(\textnormal{resp. } \frac{ \overline{\lambda_F (n)}}{n^{(k-1)/2}}\Bigr).
\]
\end{proposition}

By the above proposition, \eqref{basicr} and \eqref{basicrelation}, for any primitive Maass cusp form $f$, whether or not it is induced from a holomorphic
form, we have that
$$
\nu_f(n) = \nu_f(1)\lambda_f(n)
$$
for $n \geq 1$ and 
$
\nu_f(1) \neq 0.
$
Hence, for any fixed primitive Maass cusp form $f$, the normalized Fourier coefficients $\nu_f(n)$ for $n \geq 1$ and the Hecke eigenvalues $\lambda_f(n)$ are the same up to multiplcation by a nonzero constant.
{\em From now on, whenever we talk of primitive forms we mean 
primitive Maass cusp forms with the understanding that holomorphic modular forms are included in them}.

\subsection{The Ramanujan Conjecture}
The general Ramanujan conjecture asserts that for a primitive Maass cusp form $f \in \mathcal{C}_k (q, \chi)$ and a prime $p$, $p \nmid q$,
 the Hecke eigenvalue $\lambda_f (p)$ satisfies the bound
\begin{equation}\label{Ramu}
|\lambda_f (p)| \leq 2. 
\end{equation}
Although this conjecture is wide open, we know from the works of Kim and Shahidi, Kim, and Kim and Sarnak \cite{KS, KS1, KS2} that
\begin{align}\label{7/64}
|\lambda_f (p)| \leq 2 p^{7/64}.
\end{align}
For forms induced from holomorphic forms, the Ramanujan conjecture is a famous
theorem due to Deligne. A related conjecture concerns the size of the Laplace eigenvalues $\lambda$. Indeed, the Selberg eigenvalue conjecture, which 
says that for Maass cusp forms of weight zero, the spectral parameter $r$ should always be real (see \eqref{1/4+r2}),
can be interpreted as the Ramanujan conjecture for the infinite prime.
If Selberg's conjecture is true, then we must have $\lambda\geq 1/4$. If this is not the case, then \eqref{1/4+r2} implies that 
$r$  is purely imaginary with $|r| <1/2$. Even though we do not know the truth of the Selberg conjecture, the
 work of Kim and Sarnak cited above also gives the bound
$
| r | \leq \frac{7}{64}
$
if such exceptional eigenvalues $\lambda <1/4$ do actually occur.

\section{Moments of Hecke eigenvalues} 
For a fixed Hecke-Maass cusp form $f$, we require bounds for sums of the type 
$
\sum_{1 \leq n \leq X} |\lambda_f(n)|^{2j}.
$
 Rankin \cite{ran} and Selberg \cite{sel} had independently treated similar sums in the case of 
 holomorphic form for $j=1$. We can use standard tools of analytic number theory coupled with knowledge of analytic properties
of higher degree $L$-functions to bound such moments. 
Works of Gelbart and Jacquet \cite{GJ}, and
of Kim and Shahidi \cite{KS1}, \cite{KS2} are sufficient 
to prove the following theorem. 

\begin{theorema}\label{LauLu}
Let $f$ be a Hecke-Maass cusp form for the group ${\rm SL}(2, \mathbb{Z})$.
We have, for any $X \geq 1$ the equality 
\begin{equation}\label{L2}
\sum_{1 \leq n \leq X} |{\lambda_f(n)}|^{2} = C_f X + O_{f}(X^{3/5}),
\end{equation} 
where $C_f >0$ is a constant that depends only on the form $f$ and the same is true for the implied constant.
For $j= 2, 3$, and $4$, we have,
\begin{equation}\label{moment}
\sum_{1 \leq n \leq X} |{\lambda_f(n)}|^{2j} = X P_{f,j}(\log X)+ O_f(X^{c_{j}+\eps})
\end{equation}
for any $\eps>0$. Here $c_{j}$\rq{}s are explicit constants strictly smaller than one and $P_{f,j}$\rq{}s are polynomials of degree $ 1, 4$, and $13$ respectively and their coefficients depend on $f$. 
\end{theorema}

The first one is the well-known Rankin-Selberg estimate and a detailed proof of \eqref{moment} with explicit numerical
constants appears in
\cite{Lau-Lu}. See, in particular, Remark 1.7 and its proof at the end of the paper.
 Note that they only consider what is defined 
as a weight zero Maass cusp form here but their proof
works for general Hecke-Maass forms on ${\rm SL}(2, \mathbb{Z})$ of any weight. This can be seen
by noting that the shape of the $L$--function and the Gamma factors remain the same (see \cite[eqn. (8.17)]{DFI}) if we take the more general 
definition of Maass form as 
considered here. 
 We note the following obvious corollary of \eqref{L2} which will be required later. It can be improved slightly (by a fractional 
exponent of $\log X$) as mentioned in Remark 2 in the introduction. 
\begin{corollary}
For any Hecke-Maass cusp form $f$ for the group ${\rm SL}(2, \mathbb{Z})$, and any $X \geq 1$, we have
\begin{equation}\label{L1}
\sum_{1 \leq n \leq X} |{\lambda_f(n)}| \ll_f X,
\end{equation}
where the implied constant depends only on $f$.
\end{corollary}

\subsection{Moments of Hecke eigenvalues at primes}
The following bound on the second moment of the Hecke eigenvalues at primes
is a consequence of PNT for the Rankin-Selberg $L$-function $L(s, f \otimes f)$. See, for example, \cite[Cor. 4.2]{LIU-YE}. Similar 
results were obtained by Rankin \cite{rankin} and Perelli
\cite{Perelli} in the context of holomorphic forms.
\begin{theorema}\label{liuye}
For a Hecke-Maass cusp form $f$ for the group ${\rm SL}(2, \mathbb{Z})$, we have the bound
\begin{equation*}
\sum_{1\leq n \leq X} \Lambda (n)\,|\lambda_f(n)|^2 \ll_{f} X,
\end{equation*}
for any $X \geq 2$.
\end{theorema} 
 Note that if $f$ was a holomorphic 
form then the theorem would follow trivially from PNT and Deligne\rq{}s bound on Hecke eigenvalues. \\
From the above theorem, we deduce:
\begin{corollary} For a Hecke-Maass cusp form $f$ for the group ${\rm SL}(2, \mathbb{Z})$, we have the estimates
\begin{equation}\label{L21}
\sum_{1\leq p \leq X} |\lambda_f(p)|\log p \ll_{f} X,
\end{equation}
and
\begin{equation}\label{L22}
\sum_{1\leq p \leq X} |\lambda_f(p)| \ll_{f} X/ \log X,
\end{equation}
for any $X \geq 2$.
\end{corollary}

We also need a lower bound for the above sum and we follow the approach of 
 Holowinsky \cite[\S 4.1]{Roman} in proving the following proposition. See \cite{Ran2}, \cite{Wu}, and \cite{Wu-Xu} for 
 more precise results in this direction. 
\begin{proposition}\label{strongortho} For a Hecke-Maass cusp form $f$ of level one we have the bound
\begin{equation}\label{lowerbd}
 \sum_{1\leq p \leq X} |\lambda_f(p)| \gg_{f} X/ \log X,
\end{equation}
for all $X$ sufficiently large.
\end{proposition}

\begin{proof}
 We start with a polynomial of the form 
\[
 f(x) = c_0 +c_1(x^2 -1)+c_2(x^4-2)+c_3(x^6-5),
\]
where $c_i$'s are real, $c_0 >0$, and $f(x)$ satisfies
$
 f(x) \leq |x|
$
for all real values of $x$. For example, one can check that the polynomial
\[
 f(x) =0.01+(.09)(x^2-1)+(0.1)(x^4 -2) - (0.05)(x^6-5)
\]
satisfies all the conditions. Now, for each prime $p$, we put $x=\lambda_f(p)$ and then sum over them. The following
relations are consequences of Hecke's formula \eqref{H1}:\\
For any prime $p$, we have
\begin{align*}
 &{\lambda_f(p)}^2 -1 =\lambda_f(p^2),\\
 &{\lambda_f(p)}^4 -2 =\lambda_f(p^4)+3 \lambda_f(p^2),\\
& {\lambda_f(p)}^6 -5=\lambda_f(p^6)+5 \lambda_f(p^4)+9 \lambda_f(p^2).
\end{align*}
Now note that $\lambda_f(p^j)$ is the $p$-th coefficient of the 
$j$-th symmetric power $L$-function $L(s, \mathrm{sym}^j f)$. By facts known about symmetric power $L$-functions,
it follows (see, for example,   \cite[eqn. (2.23)]{Brumley}) that
\[
 \sum_{p \leq X} \lambda_f(p^j) = o(X/ \log X)
\]
as $X \longrightarrow \infty$ for $1 \leq j \leq 8$.
Therefore, by the above comments and PNT, we have the bound \eqref{lowerbd}.
\end{proof}

 \section{The Prime Number Theorem}
\subsection{Statements of the theorems}
Our goal in this section is to obtain   non-trivial bounds for the sums
$
\sum_{\substack{p \leq X }} \lambda_f(p)\chi(p)
$
and
$
\sum_{n \leq X} \mu(n)\lambda_f(n)\chi(n),
$
where $\chi$ is a Dirichlet character modulo $q$ and $f$ is a Hecke-Maass cusp form of level one. This will play an important role in the proof of the main theorem (see \S 7.1).
Recall that $\sum\limits_{n=1}^{\infty}\lambda_f(n)\chi(n)e(nz)$ is a primitive cusp form of level $q^2$, 
 provided $\chi (\textnormal{mod } q)$ is primitive (see \cite[\S 7.3]{Iw1}, \cite[Thm. 9]{Li}, and    \cite[\S 4, Remarks]{CI}). To see that the twisted
 form is an eigenfunctions of the Laplace operator, one notes that the Laplace operator commutes with the \textit{slash operator} (see \cite[\S4]{DFI}).
It is natural at this point to apply PNT for $L$-functions on ${\rm GL}(2)$ to estimate the above sums.
 A famous result due to Hoffstein and Ramakrishnan  \cite[Theorem C, part (3)]{HR} says the following.
\begin{theorema}\label{zf}
There is an effectively computable absolute constant $c>~0$ such that for any primitive form $f$ of some level $q$, spectral parameter $r$, and  weight $k$, the $L$-function  $L(s, f)$ does not vanish in the region 
$$
\sigma \geq 1 - \frac{c}{\log( q(|t|+|r| +2))}.
$$
\end{theorema}
Now  \cite[Thm. 5.13]{IK}, more specifically
formula (5.52), leads to the following taking into account the absence of the exceptional zero.
\begin{theorema} Let $f$ be a primitive Maass cusp form of some level $q$, spectral parameter $r$, and weight $k$. For any $X \geq 2$, we have
\begin{equation}\label{pnt}
\sum_{\substack{p\leq X }}\lambda_f(p)\log p \ll \sqrt{q(|r|+3)}\,X\,\exp\bigl(-\frac{c}{2}\sqrt{\log X}\,\bigr),
\end{equation}
where the implied constant is absolute and $c$ is as in the previous theorem.
\end{theorema}

If $f$ is a Hecke-Maass cusp form on ${\rm SL}(2, \mathbb{Z})$ and $\chi (\textnormal{mod } q)$ is a primitive 
Dirichlet character, then
 applying the above theorem to the twisted form $f\otimes \chi$ we get the estimate
 \begin{equation}\label{Riemann}
 \sum_{\substack{p\leq X }}\lambda_f(p)\chi(p) \log p \ll q\sqrt{(|r|+3)}\, X\,\exp\bigl(-\frac{c}{2}\sqrt{\log X}\,\bigr),
 \end{equation}
where the implied constant is absolute.
Apparently, it is not possible 
to deduce from \eqref{Riemann} a similar bound for the sum 
$
\sum_{1 \leq n\leq X}\lambda_f(n)\mu(n)\chi(n)
$
by the combinatorial device presented in the proof of  \cite[Corollary 5.29]{IK}. So we shall prove from scratch
the following
theorem. 
\begin{theorem}\label{pnt2}
Let  $f$ be any Hecke-Maass cusp form for the full modular group and let $\chi (\textnormal{mod }q)$ be any Dirichlet character. Let $X \geq 2$. Then we have,
\begin{equation}\label{harpoon}
\sum_{\substack{p\leq X }}\lambda_f(p)\chi(p) \log p \ll_{f} \sqrt{q} X\exp(-c_1\sqrt{\log X})
\end{equation}
and
\begin{equation}\label{thermo}
\sum_{n\leq X}\lambda_f(n)\mu(n)\chi(n) \ll_{f} \sqrt{q} X\exp(-c_1\sqrt{\log X}),
\end{equation}
where the implied constant depends only on the form $f$ and $c_1 = \frac{\sqrt{c}}{10}$, where $c$ is the same
absolute constant that appears in Theorem \ref{zf}.
\end{theorem}

\subsection{Idea of the proof}
 
We prove the second bound \eqref{thermo} only as this is the harder one and we follow the classical method using the Perron formula and
Dirichlet series. To prove it, we need to give a good bound for the associated Dirichlet series $M(s, f \otimes \chi)$ (see 
\eqref{defM0}) in the zero-free region. This is the content of Lemma \ref{boundM}. To obtain this bound, we first relate it to 
the reciprocal of the $L$-function $L(s, f \otimes \chi)$ (see \eqref{central}). Now a suitable bound for the reciprocal of 
the $L$--function follows from a similar bound for the logarithmic derivative of the $L$-function and this is done in the proof of 
Lemma \ref{boundL-1}. Thus we are reduced to bounding the logarithmic derivative of the $L$-function which is done in the
proof of Lemma \ref{Logder} using standard techniques from complex analysis. The proof of Lemma \ref{Logder} also requires 
a uniform lower bound of the Euler factors and this is the content of Lemma \ref{uniflowLp}.
It is clear that the proof of \eqref{harpoon} will be similar and the only difference will be that instead of $M(s, f \otimes \chi)$, we shall have to work with the 
 logarithmic derivative of $L(s, f \otimes \chi)$, the required bound of which is established in Lemma \ref{Logder}.
 We prove the lemmas mentioned above in the next subsection. First 
 we introduce two notations valid for this section only. 
We shall write
$\Omega$ to denote the region in the complex plane given by
$$
\Omega =\left\{ \sigma+ i t\,: \sigma \geq 1 - \frac{c}{6\,\LL}\right\},
$$ 
and $\LL$ to denote 
\begin{equation}\label{defL}
\LL:=\log \bigl(q(\vert t\vert +|r| +2)\bigr).
\end{equation}
 
\subsection{Preparatory lemmas} 
First we start by estimating the logarithmic derivative of the $L$-function.
\begin{lemma}\label{Logder}
Let $f$ and $\chi$ be as in Theorem \ref{pnt2}. Let $c$  be the constant appearing in Theorem \ref{zf}. 
Then, for every $s\in \Omega$, we have
\begin{equation}\label{logder}
\frac{L\rq{}(s, f\otimes \chi)}{L(s, f\otimes \chi)} \ll_{f}\LL,
\end{equation}
where the implied constant depends only on the form $f$.
\end{lemma}

To prove this lemma, we first recall a consequence of the Borel-Carath\'eodory theorem (see   \cite[\S 3.9, Lemma $\alpha$]{Tit}).
\begin{lemma}\label{BC} Let $s_0\in \C$, $r>0$ and $U$ an open set containing the disk $\{s\,;\ \vert s -s_0\vert \leq r\}$. Let $M\geq 1$ and $h$ an holomorphic function on $U$, satisfying $h(s_0)\not= 0$  and the inequality
$$
\left|\frac{h(s)}{h(s_0)} \right| \leq {\rm e}^M,
$$
in the disk  $|s-s_0| \leq r$. Then, for every $s$ satisfying the inequality $\vert s- s_0\vert \leq \frac{r}{4},$  one has the inequality
$$
\Bigl\vert\,
\frac{h'(s)}{h(s)} -\sum_{\substack{\rho:h(\rho)=0\\ |s_0-\rho| \leq\frac{r}{2}}} \frac{1}{s-\rho}\, \Bigr\vert  \leq 48\,\frac{M}{r}.
$$
\end{lemma}

Now we prove Lemma \ref{Logder}. 
\begin{proof} We consider two cases separately: $\chi$ is primitive and otherwise.

$\bullet$ \textit{$\chi$ is a  primitive character.} We first suppose that $\chi$ is a primitive character modulo $q$. Then we know that 
$f\otimes \chi$ is a primitive Maass cusp form of level $q^2$. The $L$--function attached to $f$ is
$
L(s,f)=\sum_{n}\frac{\lambda_{f }(n)}{n^s}=\prod_{p} \Bigl(1- \lambda_{f}(p)p^{-s} +p^{-2s} \Bigr)^{-1},
$
and the $L$--function attached to the twisted form $f\otimes \chi$ is
\begin{equation}\label{prodeul}
L(s,f\otimes \chi)=\sum_{n}\frac{\lambda_{f}(n)\, \chi(n)}{n^s}=\prod_{p} L_{p}(s, f\otimes \chi)^{-1}
 \end{equation}
 where the local factor is
 \begin{equation}\label{defLp}
 L_{p}(s, f\otimes \chi)= \Bigl(1- \lambda_{f}(p)\, \chi (p)\, p^{-s} +\chi^2 (p)\,p^{-2s} \Bigr).
 \end{equation}
By \eqref{L1}, the infinite series and the Euler product appearing in \eqref{prodeul} 
are absolutely convergent for $\sigma >1$. We know from the theory of automorphic $L$--functions that 
the function $L(s, f\otimes \chi)$ has an analytic continuation to the whole complex plane and satisfies a functional 
 equation relating the values at $s$ and $1-s$ and has a polynomial growth 
 in the critical strip; i.e., for some  absolute constant ~$A$, one has the bound
 \begin{equation}\label{polgrow}
 L(s, f\otimes \chi) \leq {\rm e}^{A\,\LL},
 \end{equation}
 uniformly for $\sigma \geq 1/2$ (see \cite[eqn. (5.20)]{IK}).
Taking the logarithmic derivatives of \eqref{prodeul}, we have for $\sigma >1$ the equality
\begin{equation}\label{logderi}
- \frac{L\rq{}(s, f\otimes \chi)}{L(s, f\otimes \chi)}=\sum_{p}\frac{\lambda_{f}(p)\, \chi (p)\, (\log p)\,  p^{-s} -2 \chi^2(p)(\log p) \, p^{-2s}} {1- \lambda_{f}(p)\chi(p)p^{-s} +\chi^2 (p)\,p^{-2s}}.
\end{equation}
We take a point $s=\sigma+it$ in the region $\Omega$. We shall consider $t$ as fixed and develop different arguments according to the value of $\sigma$. We first assume that 
\begin{equation}\label{101/100}
\Re s=\sigma \geq 1+ \frac{1}{100}.
\end{equation}
Then, the inequality  \eqref{L1} (with $q=1$)   combined with \eqref{logderi}  easily shows
$$
\frac{L\rq{}(s, f\otimes \chi)}{L(s, f\otimes \chi)} \ll_{f} 1,
$$
  uniformly for $s$ satisfying \eqref{101/100}. We now suppose that $s$ satisfies
\begin{equation}\label{<101/100}
1+\frac{c}{10\,\LL}\leq \sigma \leq \frac{101}{100}.
\end{equation}
Since $\displaystyle{\frac{L\rq{}(s, f\otimes \chi)}{L(s, f\otimes \chi)}}$ converges absolutely in the region  $\Re s>1$ (see  \eqref{L1} \&   \eqref{logderi})  we expand it in Dirichlet series:
\begin{equation}
-\frac{L\rq{}(s , f\otimes \chi)}{L(s , f\otimes \chi)} =\sum_{n \geq 1}\frac{\Lambda_{f\otimes \chi}(n)}{n^{s }}
\end{equation}
(see \cite[(5.25)]{IK}). The support of the function $\Lambda_{f\otimes \chi}$ is included
in the set of powers of primes. We deduce the inequality
$$
\left|-\frac{L\rq{}(s , f\otimes \chi)}{L(s , f\otimes \chi)}\right| \leq
\sum_p \frac{\vert \lambda_f (p)\vert \log p}{p^{\sigma }} +O(1),
$$
 the contribution from the higher powers of primes being absorbed in the $O(1)$ term thanks to the Kim-Sarnak bound \eqref{7/64}. Applying (\ref{L21}) to the above sum via partial summation,  we get the inequalities
\begin{align}
-\frac{L\rq{}(s , f\otimes \chi)}{L(s , f\otimes \chi)} &\ll_f \frac{1}{\sigma  -1}+1\notag\\
&\ll_{f} \LL \label{logder2},
\end{align}
uniformly for $s$ satisfying \eqref{<101/100} and thus the bound \eqref{logder} for $s$ in that region.

The imaginary part $t$  being fixed all the time, we consider the three points
\begin{equation}\label{ss1s0}
s=\sigma +it, 
s_1= 1+\frac{c}{10\,\LL} +it, 
s_0 =\frac{101}{100} +it,
\end{equation}
where $\sigma$ satisfies
\begin{equation}\label{ineqsigma}
 1-\frac{c}{6\,\LL}\leq \sigma < 1+\frac{c}{10\,\LL}:=\sigma_{1}.
\end{equation}
 We plan to apply Lemma \ref{BC} twice  to the function $h(s) = L(s , f\otimes \chi)$ at the point  $s_0$ and $r=1/2$.   Note that, uniformly over $t$, one has $h(s_0) \asymp 1$ by the Dirichlet series and the Euler product expression \eqref{prodeul}.
  By \eqref{polgrow}, we can choose  
  $
M \ll \LL,
$
where the implied constant is absolute. So we can write the two equalities
 \begin{equation}\label{zerosums}
\frac{L\rq{}(s, f\otimes \chi)}{L(s, f\otimes \chi)} = \sum_{\substack{|s_0-\rho| <1/4\\ L(\rho,f\otimes \chi)=0}}\frac{1}{s-\rho}+ O( \LL),
\end{equation}
and 
\begin{equation}\label{zerosums1}
\frac{L\rq{}(s_1, f\otimes \chi)}{L(s_1, f\otimes \chi)} = \sum_{\substack{|s_0-\rho| <1/4\\ L(\rho,f\otimes \chi)=0}}\frac{1}{s_1-\rho}+ O( \LL),
\end{equation}
since we have $\vert s_1 -s_0\vert \leq \vert s-s_0\vert \leq 1/8$. Subtracting \eqref{zerosums} from \eqref{zerosums1} and using \eqref{logder2} (at the point $s_1$) we deduce the equality 
 $$
\frac{L\rq{}(s, f\otimes \chi)}{L(s, f\otimes \chi)} = \sum_{\substack{|s_0-\rho| <1/4\\ L(\rho,f\otimes \chi)=0}}\frac{1}{s-\rho}  \ -  
\sum_{\substack{|s_0-\rho| <1/4\\ L(\rho,f\otimes \chi)=0}}\frac{1}{s_1-\rho}+O_{f} (\LL).
$$
Moreover,  we have the inequalities
\begin{align*}
\frac{1}{s-\rho} - \frac{1}{s_1-\rho}& \ll \frac{|s-s_1|}{|s -\rho|^2}\\
& \ll \frac{1}{\LL\,|s -\rho|^2}\\
&\ll \Re \frac{1}{s_1-\rho},
\end{align*}
since,  by Theorem \ref{zf} and the definitions \eqref{ss1s0}, we have the inequalities
$$
\vert s-\rho\vert \gg \vert s_1 -\rho \vert \textit{ and } \Re (s_{1}-\rho) \gg \LL^{-1},$$
valid uniformly.
Now we sum over the zeros $\rho$ of $L(s, f\otimes \chi)$ with $|s_0-\rho| <1/4$ and apply  (\ref{zerosums1}) and (\ref{logder2}) again to obtain 
\begin{equation}\label{11:49}
\frac{L\rq{}(s, f\otimes \chi)}{L(s, f\otimes \chi)} \ll_{f}\mathcal{L}.
\end{equation}
This gives \eqref{logder} when $\chi$ is primitive.

$\bullet$ \textit{ $\chi$ is not primitive.}
 We suppose that the Dirichlet character $\chi$ modulo $q$ is induced by 
a primitive character $\chi^*$ modulo $q^*$. From the equality
$$
L(s,f \otimes \chi)=L(s, f\otimes \chi^*)\prod_{p\mid q, \, p\nmid q^*}L_p (s,f\otimes \chi^*),
$$
we deduce the following equality between logarithmic derivatives
$$ - \frac{L\rq{}(s, f\otimes \chi)}{L(s, f\otimes \chi)}
= - \frac{L\rq{}(s, f\otimes \chi^*)}{L(s, f\otimes \chi^*)} +O\Bigl( \sum_{p\mid q,\, p\nmid q^*}\frac{\vert \lambda_{f}(p)\vert \log p}{p^\sigma}\Bigr)+O(1),
$$
where, for the second term on the right hand side, we use a uniform lower bound
for $\vert L_p (s,f\otimes \chi^*)\vert$ for $\sigma \geq  99/100$ and this will be proved in Lemma \ref{uniflowLp} below.
Using \eqref{7/64} once more, we have the equality
\begin{align*}
- \frac{L\rq{}(s, f\otimes \chi)}{L(s, f\otimes \chi)}&= - \frac{L\rq{}(s, f\otimes \chi^*)}{L(s, f\otimes \chi^*)} +O\Bigl(\sum_{p\mid q}p^{-\frac{3}{4}}\Bigr)+O (1)\\
 &= - \frac{L\rq{}(s, f\otimes \chi^*)}{L(s, f\otimes \chi^*)}+O\bigl( \log^\frac{1}{4} (q+1)\bigr),
\end{align*}
uniformly for $\sigma\geq 99/100$.
Combining with \eqref{11:49}, we complete the proof of Lemma \ref{Logder} in all the cases.
\end{proof}

\subsection{Bounds for $L$ and $L^{-1}$ inside $\Omega$}
From Lemma \ref{Logder}, we now deduce upper bounds for $L$, $L^{-1}$ and some allied functions  inside $\Omega$.
\begin{lemma}\label{boundL-1} Under the conditions of Lemma \ref{Logder}, we have the uniform bound
$$
L(s, f\otimes \chi)\text{ and }L^{-1}(s, f\otimes \chi) \ll_f \LL,
$$
for all $s\in \Omega$ where the implied constant depends only on $f$.
\end{lemma}
\begin{proof} Let  $s $ and $s_1$ as in \eqref{ss1s0} and we first suppose that $\sigma$ satisfies \eqref{ineqsigma}. Integrating the bound given
by Lemma \ref{Logder} between $s_{1}$ and $s$, we obtain
the inequality
\begin{equation}\label{difference1}
\log L(\sigma_1+ it, f\otimes \chi) - \log L(\sigma+it, f\otimes \chi) \ll_{f} 1. 
\end{equation}
To bound $\vert L(s_1, f\otimes \chi)\vert$ from above, we use the Dirichlet series expression \eqref{prodeul} to write
\begin{align}\label{720} 
\vert  L(s_1, f\otimes \chi)\vert &\leq \sum_{n\geq 1} \frac{\vert\lambda_f (n)\vert}{n^{\sigma_1}}\nonumber\\
&\ll_f \LL, 
\end{align}
using the estimate \eqref{L1}   and partial summation.

To bound $\vert L(s_1, f\otimes \chi)\vert^{-1}$ from above we introduce local factors $M_p$ defined by
\begin{equation}\label{defMp}
M_p (s, f\otimes \chi) := 1-\frac{\lambda_f (p)\chi (p)}{p^s}
\end{equation}
for each prime $p$. If
$\Re s \geq \frac{99}{100}\textit{ and } p\geq 3$, it easily follows from \eqref{7/64} that 
\begin{equation}
M_p(s, f\otimes \chi)\not= 0 \textit{ and } L_p(s, f\otimes \chi)\not= 0.
\end{equation}
However, we shall obtain a more precise statement concerning $L_p$ below; namely, Lemma \ref{uniflowLp}.\\

Write the function 
$L^{-1} $ as
\begin{equation}\label{decompL-1}
L^{-1}(s)= L_2 (s)\  \Bigl(\,\prod_{p\geq 3}M_p (s)\, \Bigr)\ G_{\geq 3} (s),
\end{equation}
with $$
G_{\geq 3} (s):= \prod_{p\geq 3}\Bigl( L_p (s)/M_p (s)\Bigr)
$$
where we voluntarily dropped the symbol $f\otimes \chi$. Computing each of the local factors, we see that the function $G_{\geq 3} (s)$ has an expression as an infinite
product absolutely convergent for $\Re s \geq 99/100$; and hence $G_{\geq 3}$ is uniformly bounded in that region. In other words, uniformly over characters $\chi$ and for $\Re s \geq 99/100$, we have
\begin{equation}\label{defG3}
G_{\geq 3} (s) \textit { and } G_{\geq 3}^{-1} (s) \ll 1. 
\end{equation}
 For the second term in the right hand side of \eqref{decompL-1}, we may write
$$
\Bigl\vert \prod_{p\geq 3}M_p (s_1)\Bigr\vert =\Bigl\vert \,\sum_{2\nmid n}
\frac{\mu (n) \, \chi (n) \, \lambda_f (n)}{n^{s_1}}\Bigr\vert \leq 
\sum_{n\geq 1} \frac{\vert\lambda_f (n)\vert}{n^{\sigma_1}}\ll_f \LL
$$
by the multiplicativity of $\lambda_f (n)$ on squarefree integers and \eqref{720}. Furhermore, we have $\vert L_2 (s_1)\vert  \leq 3$. Gathering all these remarks into \eqref{decompL-1}, we deduce the inequality 
\begin{equation}\label{746}
\vert L^{-1}(s_1, f\otimes \chi) \vert \ll_f \LL.
\end{equation}
Now \eqref{720} and \eqref{746} yields \[\vert \log L (\sigma_1+it, f\otimes \chi)\vert \leq \log {\mathcal L} +O_f (1).\] Combining this with \eqref{difference1} we complete the proof of Lemma \ref{boundL-1} when $\sigma$ satisfies \eqref{ineqsigma}.
In the remaining case, when
 $
 \sigma > \sigma_{1},
 $
instead of using \eqref{difference1}, we merely adapt the proof of \eqref{720} and \eqref{746} as we are in the region of absolute convergence.
\end{proof}
\subsection{Extension to the $M$--function} The Dirichlet series attached to the
arithmetical function appearing in the second part of Theorem \ref{pnt2} is
\begin{equation}\label{defM0}
M(s, f\otimes \chi):= \sum_n \frac{\mu (n) \lambda_f (n) \chi (n)}{n^s}.
\end{equation}
By \eqref{L1}, we know that this series converge for $\Re s >1$. In that region, it admits an Euler product expansion
\begin{equation}\label{eull}
M(s, f\otimes \chi) =\prod_p \, M_p (s, f\otimes \chi),
\end{equation}
where $M_p (s, f\otimes \chi)$ is defined in \eqref{defMp}. The Dirichlet series $M(s, f \otimes \chi)$ is not far from
$L^{-1}(s)$. More precisely, from \eqref{decompL-1} and from \eqref{eull}, we deduce the equality
which is true for every $s\in \Omega$
\begin{equation}\label{central}
M(s, f\otimes \chi)= L_2 (s)^{-1} L^{-1} (s) M_2 (s) G_{\geq 3} ^{-1}(s).
\end{equation}
By \eqref{defG3} and Lemma \ref{boundL-1}  we control all the terms 
but the first one in the region $s\in \Omega$. Now none of the 
local Euler factor $L_{p}$ defined in \eqref{defLp} vanishes in the half plane $\{s\,:\ \Re s >1\}$, otherwise
the global $L$--function would have a pole in this region which it does not by the general theory of automorphic $L$--functions. We shall now prove a uniform lower bound for these functions $\vert L_{p}\vert $, in particular, for $p=2$. We have
\begin{lemma}\label{uniflowLp}  There is an absolute constant $C_{0} >0$ such that for any Hecke-Maass cusp form $f$ for the full modular 
group, any Dirichlet character $\chi (\textnormal{mod } q)$ for any integer $q \geq 1$, any prime $p\geq 2$, and for every $s$ such that
$\Re s \geq \frac{99}{100}$, the bound
$$
\bigl\vert L_{p}(s, f\otimes \chi)\bigr \vert \geq C_{0}
$$
holds.
\end{lemma}

\begin{proof}
When $p\geq 3$, one has the inequality
\begin{align*}
\bigl\vert L_{p}(s, f\otimes \chi)\bigr \vert &\geq 1-\frac{2\cdot p^\frac{7}{64}}{p^\frac{99}{100}}-\frac{1}{p^\frac{99}{50}}\\
&\geq  1-2\cdot 3^{-\frac{1409}{1600}}-3^{-\frac{99}{50}}\\
& >1/8.
\end{align*}
by a direct application of the definition \eqref{defLp} and of the inequality \eqref{7/64}. The prime 
$2$ requires a more careful 
analysis. We write $z:= \chi (2)/2^{\sigma+it}$, $u:= \frac{1}{2} \lambda_f (2)$ and
$$
L_2(s, f\otimes \chi)=1-2uz+z^2:= G(u,z).
$$
By self-adjointness of Hecke operators, we know that the Hecke eigenvalues, in particular, $\lambda_f(2)$ and hence $u$, are real.
The existence of $C_0>0$ such that\\ $\vert L_2 (s,f\otimes \chi)\vert \geq C_0$
for all $s$ with $\Re s \geq 99/100$ is a consequence of the inequality
\begin{equation}\label{AAA}
\vert G(u, z) \vert \geq C_0,
\end{equation}
for all $(u,z)$ belonging to the set
$$
{\mathcal K}:= \{ (u,z)\in {\mathbb R}\times {\mathbb C}\,;\ \vert u\vert \leq 2^\frac{7}{64},\ \vert z \vert \leq 2^{-\frac{99}{100}}\},
$$
(by an application of \eqref{7/64}).
Since $\mathcal K$ is compact and $G$ is a continuous function, the proof of \eqref{AAA}
is reduced to the proof of the non vanishing of $G(u,z)$ on $\mathcal K$.

Let $(u_0,z_0)\in {\mathcal  K}$ satisfying $G(u_0, z_0)=0.$ We then have
$$
u_0=\frac{1}{2}(z_0+\frac{1}{z_0})=\frac{1}{2} (z_0+\frac{\overline{z_0}}{\vert z_0\vert^2}).
$$
This implies that $z_0$ is necessarily real, since $\vert z_0\vert \not=1.$

Finally, for $z$ real such that $\vert z\vert \leq 2^{-\frac{99}{100}}$, we have
$$
\Bigl\vert \frac{1}{2}(z+\frac{1}{z})\Bigr\vert \geq \frac{1}{2}(2^\frac{99}{100} +2^{-\frac{99}{100}})=1.\ 24\cdots >2^\frac{7}{64}=1.\ 07\cdots
$$
This gives a contradiction. Hence $G$ cannot vanish on $\mathcal K$ and \eqref{AAA} is proved.
The proof of Lemma \ref{uniflowLp} is now complete.
\end{proof}
 
It remains to gather in \eqref{central} the upper bounds contained in the Lemmas
\ref{boundL-1} \&  \ref{uniflowLp}, in formula \eqref{defG3}, and the bound $\vert M_2 (s)\vert \leq 3$ for $s\in \Omega$ to obtain the following.
\begin{lemma}\label{boundM} Under the conditions of Theorem 4.1, we have the bound 
$$
M(s, f\otimes \chi)  \ll_f \LL,
$$
uniformly for $s\in \Omega$.

\end{lemma}

We now have all the tools to give a sketch of the proof of Theorem \ref{pnt2}.
\subsection{Proof of Theorem \ref{pnt2}} The idea of the proof is quite standard (see for instance  \cite[Theorem 5.13]{IK}).We apply the Perron formula (see  \cite[Lem. 3.12]{Tit})
 to the Dirichlet series $M(s, f\otimes \chi)$ defined in \eqref{defM0} 
and move the contour inside the zero-free region where we can give a good estimate of  the  function $M$. We use a smoothed version of the classical Perron formula using Mellin inversion. To this end, we consider a function $\phi$ with support on $[0,X+Y]$, such that $0\leq \phi(x) \leq  1$ for $0\leq x \leq X+Y$ and $ \phi(x)  =0 $ for $x \geq  X+Y$. Here, $Y$ ($1 \leq Y \leq X/2 $) is a parameter to be chosen later. To be specific, we take 
 $$
 \phi(x) = \textnormal{min }\left(\frac{x}{Y}, 1, 1+\frac{X-x}{Y} \right)
$$
 for $0 \leq x \leq X+Y$ and 
$$
 \phi(x)=0
$$
 elsewhere. Then the Mellin transform of $\phi$ satisfies (see \cite[p.111]{IK})
 \begin{equation}\label{mellin}
 \hat{\phi}(s)  \ll \frac{X^\sigma}{|s|} \textnormal{min }\left(1, \frac{X}{|s|Y}\right)
 \end{equation}
  for $1/2\leq \Re s \leq 2$. After these preliminaries, we now give the proof of the theorem.
\begin{proof}
We have 
 \begin{multline}\label{phi}
 \sum_{n\leq X}\lambda_f(n) \chi(n)\mu(n) = \sum_{n \geq 1}\lambda_f(n) \chi(n)\mu(n)\phi(n)\\
  +
  O\Bigl(\sum_{0 <n \leq  Y}|\lambda_f(n)|\Bigr)+
O\Bigl(\sum_{X <n \leq X+Y}|\lambda_f(n)|\Bigr),
 \end{multline}
and also 
\begin{equation}\label{phi2}
\sum_{0 <n \leq  Y}|\lambda_f(n)|,    \quad  \sum_{X <n \leq X+Y}|\lambda_f(n)| \ll_{f} Y
\end{equation}
 by Cauchy\rq{}s inequality and the asymptotic formula (\ref{L2}), provided $Y \geq ~X^{\frac{3}{5}}$. 
  By the Mellin inversion formula, we can write
  \begin{equation}\label{mell}
  \sum_{n \geq 1}\lambda_f(n) \chi(n)\mu(n)\phi(n) = \frac{1}{2\pi i} \int_{(2)} M(s, f\otimes \chi)  \hat{\phi}(s)\diff s.
  \end{equation}
 Now we move the contour of the integral to the left  and deform it so that it coincides with the boundary of the region $\Omega$. Since $\Omega$ is wholly contained in the zero-free region for $L(s,f\otimes \chi)$, 
we do not encounter any pole of $M$ and thus it remains to estimate the integral over the left edge $\partial \Omega$ of $\Omega$. We assume that $q$ is not very large; namely
  \begin{equation}\label{assuforq}
  q\leq \exp (2c_{1}\sqrt{\log X}),
  \end{equation}
  otherwise, \eqref{thermo} is a trivial consequence of \eqref{L1}.  Let us write $T:=X/Y$, a parameter to be chosen later subject to $2 \leq T \leq X^{1/4}$. By   \eqref{mellin}, \eqref{mell} 
    and Lemma \ref{boundM}, we deduce the inequalities  
      \begin{align}
  \sum_{n \geq 1}\lambda_f(n) \chi(n)\mu(n)\phi(n)& \ll_f  \int_{\partial \Omega}\left|\LL \cdot  \frac{X^\sigma}{|s|} \min\left(1, \frac{X}{|s|Y}\right)\right| \diff |s|  \notag\\
& \ll_{f}  \Biggl\{\  \int_{1}^{T^2} \LL\,\frac{X^{\sigma (t)}}{t} \diff t +\int_{T^2}^\infty\LL\, \frac{X^2}{Y} \cdot \frac{1}{t^2}\diff t
\,\Biggr\}\notag \\
&\ll_f  \bigl( X^{\sigma (T^2)} +Y \bigr) \log^2 (q(T+\vert r\vert +2)),\label{contour}
  \end{align}
with 
$$
\sigma(t) :=1- \frac{c}{6{\mathcal L}}
$$
for $t$ real. For the definition of $\LL$ see \eqref{defL}.  It remains to put this in \eqref{phi}, to use \eqref{phi2}, to choose
$$
T:= \exp\bigl( 2c_1\sqrt{\log X}\bigr)),
$$
and to recall the assumption \eqref{assuforq} to finally write the inequalities
\begin{align*} 
  \sum_{n\leq X}\lambda_f(n) \chi(n)\mu(n)&
   \ll_{f}    \Bigl( X^{\sigma (T^2)} +\frac{X}{T}\Bigr) \log^2 (q(T+\vert r\vert +2))\\
 &\ll_f X\,\Bigl\{\exp\Bigl(- \frac{c\log X}{6 \log (\exp (7c_1 \sqrt{\log X}))} \Bigr) \\
 &\qquad\qquad + \exp\bigl( -2c_1 \sqrt{\log X}\bigr)
 \Bigr\}  \log^2 (q(T+\vert r\vert +2))\\
 &\ll_f X \exp\bigl( -c_1 \sqrt{\log X}\bigr),
\end{align*}
by the definition of $c_1$. This completes the proof of Theorem 
\ref{pnt2}.
 
\end{proof}


\section{Hecke multiplicative functions} 

\subsection{Hecke relation}
The following relation satisfied by Hecke eigenvalues is well known. See \cite[Chap. 14]{IK} \&  \cite[Chap. 8]{Iw2}, for instance.
 \begin{lemma}\label{Hec1} 
For every $m$ and $n\geq 1$, we have
 \begin{equation}
  \lambda_f(m)\lambda_f(n)=\sum_{d\mid (m,n)} \lambda_f \left( \frac{mn}{d^2}\right).
\end{equation}
\end{lemma}
\begin{definition}
We call a function $\lambda: \mathbb{N} \longrightarrow \mathbb{R}$ \textit{Hecke multiplicative} if $\lambda(1)=1$ and $\lambda$ satisfies the relation
\begin{equation}\label{H1}
  \lambda(m)\lambda(n)=\sum_{d\mid (m,n)} \lambda \left( \frac{mn}{d^2}\right).
\end{equation}
\end{definition}
Here we restrict ourselves to real valued functions as this is enough for our purpose and in what follows we need positivity
of $\lambda^2$.
Soundararajan \cite{Sound} had introduced a similar definition in the context of his work on the {\it Quantum Unique Ergodicity Conjecture}. 
Note that a Hecke multiplicative function is automatically multiplicative.
From (\ref{H1}), we can easily deduce the dual formula:
\begin{equation}\label{H2}
 \lambda (mn)=\sum_{d\mid (m,n)} \mu (d) \lambda\Bigl( \frac{m}{d}\Bigr) \lambda\Bigl( \frac{n}{d}\Bigr).
 \end{equation}


\subsection{The $\lambda^*$ function}

Given a Hecke multiplicative function $\lambda$, we introduce a new function $\lambda^*$ which can be thought of as an analogue of 
(square-root of) the 
 divisor function. 
\begin{definition} Let $\lambda\,:\ \N \rightarrow \R$ be an arithmetic function. We define
the arithmetical function $\lambda^*$
by declaring
\begin{equation}\label{defto}
\lambda^* (n)=\Bigl( \, \sum_{d\mid n} \lambda^2 (d)\,
\Bigr)^\frac{1}{2}\text{ for }n\geq 1.
\end{equation}
\end{definition}
Note that in the trivial case $\lambda\equiv 1$ then we have  $\lambda^*(n)=\sqrt{d(n)}$ where $d(n)$ 
is the number of postive integers
of the integer $n$.
When $\lambda$ is a Hecke multiplicative function, the associated $\lambda^*$ inherits some regularity properties
which justify its introduction.  Here are some of these.
 \begin{lemma}\label{propto} Let $\lambda$ be a Hecke multiplicative function.
Let $m$ and $n$ be any positive integers. Then the following holds.
\begin{enumerate}[(a)]
\item $\lambda^*(n) \geq 1$, \label{one}
\item  $\vert \lambda (m) \vert \leq \lambda^* (m)$,\label{23)}
\item If $m\mid n$, then $ \lambda^* (m)\leq \lambda^* (n),$\label{b)}
\item If $(m,n)=1$ then $\lambda^*(mn)=\lambda^* (m) \lambda^* (n),$\label{a)}
\item $\vert \lambda (mn)\vert \leq \lambda^* (m) \lambda^* (n),$\label{tel}
\item $\lambda^* (mn) \leq   d^\frac{1}{2}(m) \,d^\frac{1}{2}(n) \, \lambda^* (m)\, \lambda^* (n),$\label{cold}
\item  $\vert\,\lambda (m)\,\lambda (n)\,\vert \leq d^\frac{1}{2} ((m,n))\, \lambda^* (mn).$\label{f)}
\end{enumerate}
\end{lemma}

\begin{proof}
The first three assertions are trivial since $\lambda (1)=1$ and $\lambda^2 (d)\geq 0$, for all $d$.
The part \eqref{a)} is a consequence of the fact that if $d\mid mn$, then $d$ can be uniquely written as $d=d_{1}d_{2}$ where $d_{1}$ and $d_{2}$ respectively divide $m$ and $n$.  We also use the relation $\lambda (ab)= \lambda (a) \lambda (b)$, when $a$ and $b$ are coprime.
For the part \eqref{tel}, we use \eqref{H2} to write 
$$
\vert \lambda (mn) \vert \leq \sum_{d\mid (m,n)} \bigl\vert \lambda \bigl( \frac{m}{d}\bigr)\, \lambda \bigl(\frac{n}{d}\bigr) \, \bigr\vert \leq \Bigl( \sum_{d\mid (m,n)} \lambda^2 \bigl( \frac{m}{d}\bigr)\Bigr)^\frac{1}{2}\cdot \Bigl( \sum_{d\mid (m,n)} \lambda^2 \bigl( \frac{n}{d}\bigr)
\Bigr)^\frac{1}{2},
$$
hence the result by extending summation.
In the case of \eqref{cold}, we write
  \begin{align*}
{\lambda^*}^2 (mn)&=\sum_{d\mid mn }\lambda^2
\Bigl(\frac{mn}{d}\Bigr)\\
&\leq \sum_{d_{1}\mid m}\ \sum_{d_{2}\mid n}\lambda^2 \Bigl( \frac{m}{d_{1}}\cdot \frac{n}{d_{2}}\Bigr)\\
&\leq \sum_{d_{1}\mid m}\ \sum_{d_{2}\mid n} {\lambda^*}^2 \Bigl( \frac{m}{d_{1}} \Bigr)\cdot {\lambda^*}^2 \Bigl( \frac{n}{d_{2}} \Bigr)\\
&\leq \sum_{d_{1}\mid m}\  {\lambda^*}^2 (m)\cdot  \sum_{d_{2}\mid n}{\lambda^*}^2 (n)\\
&\leq d(m) \,d(n)\,{ \lambda^*}^2(m)\,  {\lambda^*}^2 (n),
\end{align*}
by \eqref{tel} and \eqref{b)}.  
For  \eqref{f)}, we write by \eqref{H1}, the inequalities
\begin{align*}
\vert\, \lambda (m)\, \lambda (n )\, \vert&  \leq \sum_{d\mid (m,n)} \Bigr\vert \lambda\Bigl( \frac{mn}{d^2}\Bigr)\Bigr\vert\\
&\leq \Bigl( \sum_{d\mid (m,n)} 1 \Bigr)^\frac{1}{2}\cdot
\Bigl(\, \sum_{d\mid (m,n)}   \lambda^2\Bigl( \frac{mn}{d^2}\Bigr)\,\Bigr)^\frac{1}{2}\\
&\leq d^\frac{1}{2} ((m,n)) \lambda^* (mn),
\end{align*}
by Cauchy-Schwarz inequality and extending summation.
\end{proof}


\subsection{Moments of $\lambda^*(n)$}

 The divisor function $d(n)$ satisfies nice bounds if we sum its powers over an interval. Indeed, for any positive integer $A$,
we have, for $X \geq 1$,  
\begin{equation}\label{divmoment}
\sum_{n \leq X} d^A(n) \ll_{A} X(\log X)^{2^A -1}.
\end{equation}
For this classical bound  see \cite[p.61]{MoVa} for instance.
The function $\lambda^*$ also displays similar regularity and it is reasonable to expect that moments of $\lambda^*$ should be of same size as corresponding moments of $\lambda$ (up to $\log$ factors).
With a specific application in mind, we prove a particular case of this regularity.


\begin{proposition}\label{6thmoment}
Suppose a Hecke multiplicative function $\lambda$ satisfies the bound
\begin{equation}\label{assum}
 \sum_{m\leq M} {\lambda}^6 (m) \ll_{\lambda}  M (\log M)^4
\end{equation}
uniformly for all $M\geq 2$.
Then for any positive integer $A$, there is some integer $A_1= A_1 (A)$, such that, uniformly for $X\geq 2$, one has the estimate
\begin{equation}\label{4h}
\sum_{m\leq X} d^A (m) {\lambda^*}^4 (m) \ll X\, (\log X)^{A_1},
\end{equation}
where the implied constant depends only on $\lambda$ and $A$.
\end{proposition}

\begin{proof} 

Throughout the proof we denote by $A_1$ some unspecified but effective function of $A$. The value of $A_1$ may be different in different occurrences.  By the definition \eqref{defto} of the function $\lambda^*$, one has the equality
\begin{align}
\sum_{m\leq X}d^A (m)\,{\lambda^*}^4 (m)&=\sum_{m\leq X}\ d^A (m)\,\Bigl( \,\sum_{d\mid m}\,\lambda^2 (d)\,\Bigr)^2\nonumber\\
&=\underset{d_1,\ d_2}{\sum\ \sum } \lambda^2 (d_1) \lambda^2(d_2)\, \sum_{\substack{m\leq X\\ [d_1,d_2] \mid m} }d^A (m),\label{green}
\end{align}
where $[d_{1},d_{2}]$ is the least common multiple of $d_{1}$ and $d_{2}$.  Using the inequality
\begin{equation}\label{sun0}  
d(ab)\leq d(a) d(b),
\end{equation}
and \eqref{divmoment}, we transform \eqref{green} into
$$
\sum_{m\leq X}d^A (m)\,{\lambda^*}^4 (m)\ll \LL^{A_1}
\underset{d_1,\ d_2}{\sum\ \sum } d^A (d_1)\,d^A (d_2)\, \lambda^2 (d_1) \lambda^2(d_2)\, \frac{X}{[d_1,d_2]},
$$ 
where $\LL$ has now the meaning $$\LL := \log 2X.$$
Since $[d_{1}, d_{2}]=d_{1}d_{2}(d_{1},d_{2})^{-1} $ we extend    the summation over all the  divisors $\delta$ 
of $d_{1}$ and $d_{2}$,  to obtain  the series of inequalities
\begin{align*}
\sum_{m\leq X}d^A (m)\, {\lambda^*}^4 (m)&\ll  X\,\LL^{A_1} \,\sum_{\delta \leq  X}\delta\,  \Bigl( \sum_{\delta \mid d_{1}\leq X} d^A (d_{1})\frac{\lambda^2 (d_{1})}{d_{1}}\Bigr)^2\\
&\ll X\,\LL^{A_1}\,\sum_{\delta \leq X} \delta  \Bigl( \sum_{\delta \mid d_{1}\leq X} \frac{\lambda^4 (d_{1})}{d_{1}}\Bigr) \cdot \Bigl( \sum_{\delta \mid d_{1}\leq X} \frac{d^{2A}(d_{1})}{d_{1}}\Bigr)\\
&\leq X \, \LL^{A_1}\,\sum_{\delta \leq X}  \ d^{2A}(\delta )\,  \Bigl( \sum_{\delta \mid d_{1}\leq X} \frac{\lambda^4 (d_{1})}{d_{1}}\Bigr)\\
&\leq  X\, \LL^{A_1}\,   \Bigl( \sum_{  d_{1}\leq X} d^{2A+1}(d_{1})\, \frac{\lambda^4 (d_{1})}{d_{1}}\Bigr)\\
&\leq X\, \LL^{A_1}\,  \Bigl( \sum_{ d_{1}\leq X}   \frac{\lambda^6 (d_{1})}{d_{1}}\Bigr)^\frac{2}{3}\  \Bigl( \sum_{ d_{1}\leq X}   \frac{d^{6A+3} (d_{1})}{d_{1}}\Bigr)^\frac{1}{3}\\
&\ll X \LL^{A_1}, 
\end{align*}
where we used the Cauchy-Schwarz inequality, the inequalities \eqref{sun0} and \eqref{divmoment}, H\" older's inequality, and finally the assumption
\eqref{assum} combined with Abel summation.
\end{proof}


\section{Additive twists and Miller's theorem}

\subsection{ \textbf{GL(2)}}

 For later applications in the estimation of Type I sums  we prove the following:
\begin{lemma}\label{GL2-N} 
 Let $f$  be a cusp form  on ${\rm SL}(2,\Z)$.  Then uniformly  for $N$ integer $\geq 1$, for  $X \geq 1$ and  for $\alpha  \in \R$  one has the inequality 
$$
\sum_{n\leq X}\lambda_f (Nn)e (\alpha n) \ll_f  \sqrt{X} \log (2X)\, d (N)^\frac{1}{2}\, \lambda_{f }^* (N).
$$
 \end{lemma}
\begin{proof}
We use \eqref{H2} and \eqref{sum} to write
\begin{align*}
\sum_{n\leq X}\lambda_f (Nn)e (\alpha n)&= \sum_{d\mid N} \mu (d) \lambda_f (N/d)\sum_{k\leq X/d}
\lambda_f (k)  e (\alpha dk)\\
&\ll  \sqrt{X} (\log 2X)\,\sum_{d\mid N} \mu^2 (d)\, \bigl\vert \lambda_f (N/d)\bigr\vert d^{-\frac{1}{2}}.
\end{align*}
   It remains to apply
the Cauchy--Schwarz  inequality and to refer to the definition \eqref{defto} to conclude the proof.
\end{proof}


\subsection{ \textbf{GL(3)}}
 
We recall the main theorem in \cite{Mi} already mentioned in \eqref{Mill100} above. Miller's theorem depends 
crucially on the Voronoi summation formula for ${\rm GL}(3)$ which was first established by Miller and Schmidt \cite{MS1}  (see also \cite{GL1} for a different treatment). A concrete introduction to the theory of higher degree automorphic forms 
is the book \cite{GL}.

\begin{theorema}\label{MILLL}
Let $a_{r,n}$ denote the Fourier coefficients of a cusp form $f$  on \break  ${\rm GL}(3,\Z)\backslash {\rm GL}(3,\R)$. Then for every $\eps >0$, for every integer $r$, and for every $T\geq 1$, one has the inequality 
$$
\sum_{n\leq T} a_{r,n}  e(n\alpha) \ll_{f,r,\eps} T^{\frac{3}{4}+\eps},
$$
 where the implied constant depends only on the form $f$, $r$, and $\eps$.
\end{theorema}
 
Applying this theorem to the symmetric square lift of a Hecke-Maass cusp form $f$ of level one and noting that 
we can write the coefficients of $L(s, {\rm sym}^2 f)$ as convolutions from the expression 
$$
 L(s, {\rm sym}^2 f)=\zeta (2s) \sum_{n=1}^\infty \frac{\lambda_f (n^2)}{n^s},
 $$
we obtain the following corollary. See \cite[p. 434--435]{MSsurvey} for details. 

\begin{corollary}\label{165}
For every  Hecke-Maass cusp form $f$ of level one and for every $\varepsilon >0$, there exists a function $C(f,\varepsilon)$ such that, for every  $T\geq 1$  one has the inequality
$$
\Bigl\vert \sum_{n\leq T} \Bigl( \sum_{n=md^2}\lambda_f (m^2)\Bigr)e (n \alpha)\Bigr\vert \leq C(f,\varepsilon) T^{\frac{3}{4} +\eps}.
$$
\end{corollary}

\subsection{Application of Miller's theorem} We have 

\begin{lemma}\label{Miller}

For every a Hecke-Maass cusp form $f$ of level one and for every postive $\eps$, we have
$$
\sum_{n\leq T} \lambda_f (n^2)e (n\alpha) \ll_{\eps,f} T^{\frac{3}{4} +\eps},
$$
uniformly for $T\geq 1$.
\end{lemma} 
\begin{proof}
Let 
$$
 M(T,\alpha):=\sum_{n\leq T} \Bigl( \sum_{n=md^2}\lambda_f (m^2)\Bigr)e (n \alpha),$$ 
and let 
$$
S(T, \alpha):= \sum_{n\leq T} \lambda_f (n^2)e (n\alpha).
$$
We claim the equality
\begin{equation}\label{193}
S(T,\alpha) = \sum_{r\leq \sqrt T} \mu (r) M\Bigl( \frac{T}{r^2}, r^2 \alpha \Bigr),
\end{equation}
and Lemma \ref{Miller} directly follows from Proposition \ref{165} after a summation over $r$. To prove \eqref{193}, we write 
\begin{align*}
S(T, \alpha)&= \sum_{m\leq \sqrt T} \Bigl( \sum_{r\mid m}\mu (r) \Bigr) S\Bigl( \frac{T}{m^2}, m^2 \alpha\Bigr)\\
&=\sum_{r\leq \sqrt T} \mu (r)\sum_{\ell \leq \sqrt T/r} \ \sum_{k\leq T/(\ell^2 r^2)}\lambda_f (k^2) e(kr^2\ell^2 \alpha).
\end{align*}
The proof now follows by making the change of variables $n=k\ell^2$.
\end{proof}

Let us denote, for positive integer  $A$,
$$
S(T,A, \alpha) :=  \sum_{n\leq T} \lambda_f (A n^2)e (n\alpha).
$$
By \eqref{H2} and the observation that for a squarefree $\ell$, $\ell \mid n^2$ if and only if $ \ell \mid n,$ we have
\begin{equation}\label{M1}
S(T,A,\alpha)=\sum_{\ell \mid A}\,\mu (\ell)\,\lambda_f(A/\ell)\,S(T/\ell , \ell, \ell \alpha),
\end{equation}
for any integer $A$.
Now we prove a key lemma.

\begin{lemma}\label{key} Let $f$ be a Hecke-Maass cusp form  of level one and let $\eps$ be any positive real number.  Then we have
the bound
\begin{equation}\label{keyeq}
S(T,A,\alpha)   \ll_{\eps,f} (1+\omega(A))\,d_3(A)\,|\lambda_f(A)|\,T^{\frac{3}{4}+\eps},
\end{equation}
uniformly  for $T\geq 1$,  for  squarefree $A \geq 1$ and for  real $\alpha$.
 
\end{lemma}

\begin{proof}

We shall prove this Lemma for every squarefree $A$ by induction on $T$,
with the same implicit constant as the one contained in the statement of Lemma 6.2.
For $T_0\leq 1$, formula (73) is correct for any $A$. Similarly, (73) is correct for any
$T$ when $A=1$,
with the same constant as in Lemma 6.2.
Suppose now that there exists $T_0 \geq 1$, such that (73) is true for any $T\leq T_0$
and any $A$
squarefree. We now prove that the same holds for any $T\leq 2T_0$.

We start with the relation (\ref{M1}). The first term corresponding to $\ell=1$ is $\lambda_f(A) S(T, 1, \alpha)$ and $S(T, 1, \alpha) = O_\eps 
(T^{\frac{3}{4}+\eps})$ by Lemma \ref{Miller}. For $\ell >1$, we use the induction hypothesis. Since $A$ is squarefree, for $\ell \mid A$  we 
have $(\ell, A/\ell) =1$ (hence $\lambda_f (A) =\lambda_f(A/\ell) \lambda_f (\ell)$)  and also  $1+\omega(\ell) \leq \omega(A)$ for $\ell \neq 
A$. Thus we have,
\begin{align*}
S(T, A, \alpha) &\ll  \bigl\vert  \lambda_f(A)\bigr\vert \,T^{\frac{3}{4}+\eps}\Biggl\{1+\omega(A)\sum_{\substack{\ell \mid A\\1 <\ell <A}}|\mu(\ell)|\frac{d_3(\ell)}{\ell^{\frac{3}{4}+\eps}}\\
& \qquad\qquad\qquad\qquad\qquad\qquad\quad+\frac{d_3(A)(1+\omega(A))}{A^{\frac{3}{4}+\eps}} \Biggr\}\\
&\leq \bigl\vert \lambda_f(A)\bigr\vert \,T^{\frac{3}{4}+\eps}\Biggl\{1+\omega(A)\prod_{p\mid A} (1+\frac{3}{p^{\frac{3}{4}+\eps}})+\frac{d_3(A)}{A^{\frac{3}{4}+\eps}}\Biggr\}.
\end{align*}
Now we note that 
$$
\prod_{p\mid A} (1+\frac{3}{p^{\frac{3}{4}}}) < d_3(A),
$$
as $3/p^{\frac{3}{4}} <2$ for all primes $p$. Since we also have $1+d_3 (A)/A^{3/4} < d_3(A)$ for all $A \geq 2$, we deduce 
\begin{align*}
S(T, A, \alpha) &\ll \lambda_f(A)T^{\frac{3}{4}+\eps}\left\{1+\omega(A)d_3(A)+\frac{d_3(A)}{A^{\frac{3}{4}}}\right\}\\
& \ll  (1+\omega(A))\,d_3(A)\,\lambda_f(A)\,T^{\frac{3}{4}+\eps}.
\end{align*}

\end{proof} 

Now we generalize this to all integers $A$, squarefree or not, by  using the function $\lambda_f^*$ defined in \eqref{defto}.

\begin{lemma}\label{Millgen}
Let $f$ be a Hecke-Maass cusp form of level one  and let $\eps >0$ be any real number. Then we have
the inequality
$$
S(T,A,\alpha) \ll_{ \varepsilon,f} \bigl(1+\omega (A)\bigr)\, d_{3}(A) \, {\lambda_f ^*}^2(A) \, T^{\frac{3}{4} +\varepsilon},
$$
uniformly  for $T\geq 1$, for $A\geq 1$ and for real $\alpha$.
\end{lemma}
\begin{proof}  We start from \eqref{M1}. Applying \eqref{keyeq}, we obtain 
\begin{align*}
S(T,A,\alpha)& \ll \bigl(1+\omega (A)\bigr)\, d_{3}(A)\,T^{\frac{3}{4} +\varepsilon}\sum_{\ell \mid A}\vert \lambda_f (\ell)\vert\, \vert \lambda_f (A/\ell)\vert \\
&\ll \bigl(1+\omega (A)\bigr)\, d_{3}(A)\,{\lambda_f^*}^2 (A)\,T^{\frac{3}{4} +\varepsilon},
\end{align*}
by the Cauchy-Schwarz inequality and the definition \eqref{defto}.
\end{proof}


\section{The proof of Theorem \ref{mainthm1}}\label{section5}
We assume throughout the rest of the paper 
that $f$ is a Hecke-Maass cusp form of level one. There 
is no loss of generality in doing so as the space of Maass cusp forms is spanned by the Hecke forms. Recall that for such a form, the Fourier 
coefficients $\nu_f(n)$ and the Hecke eigenvalues $\lambda_f(n)$ coincide up to multiplcation by the non-zero constant $\nu_f(1)$.
We prove only the bound \eqref{144} for the sum  involving the M\" obius  function. The proof of the   bound \eqref{139}  is structurally identical and, in fact, simpler as explained in the introduction. 
 Throughout the rest of the paper, $f$ denotes a Hecke-Maass cusp form for the group ${\rm SL}(2, \mathbb{Z})$.

\subsection{Initial steps}

  Let us write
$$
T(X, \alpha) = \sum_{1 \leq n \leq X} \lambda_f(n)\mu(n) e(n \alpha).
$$
We fix a parameter $Q$ to be optimized later. Now, Dirichlet\rq{}s theorem on Diophantine Approximation ensures that
given any $\alpha \in [0,1)$, there is always a rational number $a/q$, $(a, q) =1$ such that 
\begin{equation}\label{Dirichlet}
 1\leq q \leq Q \text{ and } \bigg| \alpha - \frac{a}{q} \bigg| \leq \frac{1}{qQ}.
\end{equation}
By partial summation, we have
\begin{equation}\label{ps}
\Bigl\vert\, T(X, \alpha) \Bigr\vert \ll   \Bigl\vert\, T\biggl(X, \frac{a}{q}\biggr) \,\Bigr\vert  + \int_1^X \bigg|\biggl(\alpha - \frac{a}{q}\biggr)T\biggl(x, \frac{a}{q}\biggr)\bigg|\diff x +1
\end{equation}
We now plan a general study of the sum $T(x,a/q)$.
We first write the equality
\begin{equation}\label{split1}
T\biggl(x, \frac{a}{q}\biggr) = \sum_{b (\textnormal{mod } q)}e\biggl(\frac{ab}{q}\biggr)\sum_{\substack{n \equiv b (\textnormal{mod } q)\\n \leq x}} \lambda_f(n)\mu(n).
\end{equation}
To detect the congruence $n\equiv b\bmod q$ by Dirichlet characters, we must first ensure the coprimality
of the class and the modulus. So we introduce
\[
d=(b,q),\ b_{1}=b/d, \  q_{1}=q/d, \textit{ and } 
\]
\[\chi_{d}, \textit{ the principal character modulo } d.\]
This gives the equalities
\begin{align}
\sum_{\substack{n \equiv b (\textnormal{mod } q)\\n \leq x}} \lambda_f(n)\mu(n)
&=
\sum_{\substack{n_{1} \equiv b_{1} (\textnormal{mod } q_{1})\\n_{1} \leq x/d}} \lambda_f(dn_{1})\mu(dn_{1})\notag \\
&=\lambda_{f }(d) \mu (d) \sum_{\substack{n_{1} \equiv b_{1} (\textnormal{mod } q_{1})\\n_{1} \leq x/d}} \lambda_f(n_{1})\mu(n_{1})\chi_{d }(n_{1})\notag\\
&=\frac{\lambda_{f }(d) \mu (d)}{\varphi (q_{1})} \ \sum_{\chi (\textnormal{mod } q_{1})}\ \overline{\chi} (b_{1})\, \sum_{ n_{1} \leq x/d} \lambda_f(n_{1})\mu(n_{1})\bigl( \chi \chi_{d }\bigr)(n_{1}).
\end{align}

Since $\chi \chi_d$ is a character of modulus $d q_1$, we can apply Theorem \ref{pnt2} with $q:= dq_{1}$ to the inner sum. This gives 
\[
 \sum_{ n_{1} \leq x/d} \lambda_f(n_{1})\mu(n_{1})\bigl( \chi \chi_{d }\bigr)(n_{1}) \ll \sqrt{d q_1} \frac{X}{d} \exp\left(-c_1 \sqrt{\log (X/d)}\right).
\]
Bounding $\lambda_f(d)$ by \eqref{7/64}, we finally have 
\begin{equation}\label{bengal}
T\biggl(x, \frac{a}{q}\biggr) \ll q^{3/2}X \exp\left(-c_1 \sqrt{\log (X/q)}\right).
\end{equation}

Now the proof will proceed differently depending on the size of $q$ compared to $X$.

\subsection{Major arcs}

By \eqref{bengal}, \eqref{ps}, and \eqref{Dirichlet}, we have 
\begin{align}
T(X, \alpha) &\ll  q^{3/2} X \exp\left(-c_1\sqrt{\log (X/q)}\right)\left(1+\left|\alpha - \frac{a}{q}\right|\,X \right) \notag\\
& \ll  \sqrt{q}\, X \exp\left(-c_1\sqrt{\log (X/q)}\right)\left(q+\frac{X}{Q}\right).\label{Gauss}
\end{align}
Now we choose
\begin{equation}\label{qchoice}
Q = X\exp\left(-\frac{c_1}{3}\sqrt{\log X}\right).
\end{equation}
If 
\begin{equation}\label{smalldenom}
q \leq \frac{X}{Q} = \exp\left(\frac{c_1}{3}\sqrt{\log X}\right),
\end{equation}
then by \eqref{Gauss},
$$
T(X, \alpha) \ll X \exp\left(-\frac{c_1}{10}\sqrt{\log X}\right).
$$
Therefore, we have proved Theorem \ref{mainthm1} if $\alpha$ admits a good enough 
rational approximation $a/q$; $(a, q)=1$, satisfying (\ref{Dirichlet}) with $Q$ is as above and $q$ satisfies the bound \eqref{smalldenom}. 
On the other hand, if $\alpha$ is such that \eqref{smalldenom} is true for no rational number $a/q$; $(a, q)=1$, satisfying (\ref{Dirichlet}),
then this method does not work and we apply the Vinogradov method as explained in the next few subsections.


\subsection{Minor arcs}

 After the pioneering work of Vinogradov, Gallagher, Vaughan and others, we know how to quickly  enter into the combinatorial structure of  the functions $\Lambda$ and $\mu$. 
In our situation we use  (see \cite[Prop.13.5]{IK} for instance):
 
\begin{proposition}
Let $y, z \geq 1$. Then for any $m > \max\{y, z\}$, we have
 \begin{equation} \label{vaughan}
 \mu(m) = -\sum_{\substack{bc\mid m\\b \leq y, c \leq z}} \mu(b) \mu(c) + \sum_{\substack{bc\mid m\\ b>y, c > z}} \mu(b) \mu(c).
 \end{equation}
 \end{proposition}
 Accordingly, we decompose the sum $T(X, \alpha)$ as 
 \begin{equation}\label{types}
 T(X, \alpha) =- T_1 (X, \alpha) + T_2 (X, \alpha) +O (y+z),
 \end{equation}
where 
\begin{equation}\label{t1}
T_1 (X, \alpha) = \sum_{b \leq y} \mu(b) \sum_{c \leq z} \mu(c) \sum_{k \leq X/bc} \lambda_f(kbc)e(kbc \alpha),
\end{equation}
and 
\begin{equation}\label{t2}
T_2 (X, \alpha) = \sum_{b > y} \mu(b) \sum_{c > z} \mu(c) \sum_{k \leq X/bc} \lambda_f(kbc)e(kbc \alpha)
\end{equation}
are called {\it sums of type I and type II}   respectively. The parameters $y\geq 1$ and $z\geq 1$ will be chosen later optimally (they will be
 of size about
$O(X^{1/5})$).
  The error term in \eqref{types} comes from the contribution of the $m\leq \max \{y,z\}$ and is handled with the inequality \eqref{L1}.


\subsection{Type I sums}

 A direct application of Lemma \ref{GL2-N} to the inner sum of  \eqref{t1} leads to the upper bound
$$
\sum_{k \leq X/bc} \lambda_f(kbc)e(kbc \alpha) \ll_{\eps} (bc)^{2\eps}\,\lambda_f^*(bc)\left(\frac{X}{bc}\right)^{\frac{1}{2}+\eps},
$$
after using standard bounds for the arithmetical functions involved. Inserting this bound in \eqref{t1} and writing $m:=bc$ we obtain
\begin{align}
T_1 (X, \alpha) &\ll_{\eps} X^{\frac{1}{2}+\eps}(yz)^{\eps} \sum_{m \leq yz}\frac{d (m) \lambda_f^*(m) }{m^{\frac{1}{2}}}  \notag\\
 & \ll_{\eps}(Xyz)^{\frac{1}{2}+\eps}\label{type1}
\end{align}
by Theorem \ref{LauLu},  Proposition \ref{6thmoment}, and partial summation.


\subsection{Type II sum}

Now we come to the most delicate part of the proof which is the estimation of the type II sum.
 Introducing the notation 
$$
\beta_{\ell} := \sum_{\substack{b|\ell \\b>y}} \mu(b),
$$
we see that
$$
T_2 (X, \alpha)=\sum_{\ell} \beta_{\ell}\sum_{\substack{c>z\\\ell c\leq X}}\mu(c)\lambda_f(c\ell)e(\alpha c\ell)
$$
and $\beta_{\ell}$ satisfies the bound

\begin{equation}\label{divisor}
|\beta_{\ell}| \leq d(\ell).
\end{equation}

Now we introduce two parameters $L$ and $C$ which will be chosen later subject to 
\begin{equation}\label{condLC}
L > y,\,  C > z \textit{ and  }LC \leq X.
\end{equation} 
We split the sum
$T_2 (X, \alpha)$ into $O((\log X)^2)$  many dyadic pieces of the form 
$$
T_2(C, L, \alpha) = \sum_{\ell \sim L} \beta_{\ell}\sum_{\substack{c \sim C}}\mu(c)\lambda_f(c\ell)e(\alpha c\ell),
$$
where the variables $\ell $ and $c$ satisfy the extra condition
\begin{equation}\label{<X}
c\ell \leq X.
\end{equation}
This extra condition is sometimes superfluous but allows us to suppress the dependence on $X$ in the notations. By the Cauchy-Schwarz inequality we have 
 
\begin{equation}\label{type2cauchy}
\left|T_2(C, L, \alpha)\right|^2 \leq\left( \sum_{\ell \sim L} |\beta_{\ell}|^2\right) A(C, L, \alpha),
\end{equation}
where 
$$
A(C, L, \alpha) := \sum_{\ell \sim L }\,\Bigl\vert\, \sum_{c\sim C} \mu (c) \lambda_f(c\ell) e\bigl( \alpha c \ell)\,
\Bigr\vert^2,
$$
with the extra constraint \eqref{<X}.
By \eqref{divisor} and \eqref{divmoment},
\begin{equation}\label{betamoment}
\sum_{\ell \sim L} |\beta_{\ell}|^2\ \ll L (\log 2L)^3,
\end{equation}
where the implied constant is absolute. Now it remains to estimate $A(C, L, \alpha)$ and we can give a non-trivial bound as  long as $\alpha$ is not close to rationals with small denominators. Precisely we prove the following.
\begin{theorem}\label{centrall} Let $f$ be a Hecke-Maass cusp form of level one.  Suppose $\alpha$ is a real number and $a/q$ is any rational number written as a reduced fraction such that
\begin{equation}\label{Dir}
\bigl\vert \alpha -\frac{a}{q} \bigr\vert \leq \frac{1}{q^2}.
\end{equation}
Then there are absolute constants $K$ and $K' >0$, and for all $\eps >0$ a constant $C(\eps)$, such that 
\begin{align*}
A(C,L,\alpha) \leq C(\eps)\,   C^2 L^\frac{5}{6} (CL)^\eps + K' \bigl( C^\frac{3}{2}L + C^2L q^{-\frac{1}{2}}
+C^\frac{3}{2}L^\frac{1}{2}q^\frac{1}{2}\bigr) (\log (2CL))^K,
\end{align*}
uniformly for all $C$, $L$ and $X>1$.
\end{theorem}

\noindent
\textit{Remark.} To test the strength of  Theorem \ref{centrall}, we first give a trivial bound of $A(C,L,\alpha)$. By $\LL$, we 
now denote 
$$\LL:=\log 2CL \ (\ll \log 2X).
$$ 
We have
\begin{align*}
 A(C,&L,\alpha) \leq C \sum_{\ell \sim L } \sum_{c \sim C}\vert \lambda_f (c\ell) \vert^2\\
 &\ll C\,\sum_{CL<m\leq 4CL} d(m) \lambda_f^2 (m) \\
&\ll_f C^2L\,{\LL}^2,
\end{align*}
by Cauchy's inequality, \eqref{moment} and \eqref{divmoment}. Hence the theorem is useful if we have $C$, $L$ and $q$ satisfy the inequalities: $C$ and   $L\geq (CL)^\eps $ and
  ${\LL}^{ K_1} \leq q \leq (CL){\LL}^{- K_1}$, where $K_1$ is an explicit constant.
Now we give a proof of Theorem \ref{centrall}.

\begin{proof}

Throughout the proof, $K$ will denote an unspecified but effective constant the value of which may change in different occurrences.
Expanding square and inverting summations, we can write
\begin{equation}\label{scis}
A (C,L,\alpha) =\underset{c_{1},\,c_{2}\sim C}{\sum \ \sum} \,\mu (c_{1}) \mu (c_{2}) \sum_{\ell \sim L}
\lambda_f(c_{1}\ell) \,\lambda_f(c_{2}\ell)\, e\bigl( \alpha (c_{1}-c_{2})\ell \bigr),
\end{equation}
where $\ell$ satisfies the extra inequality
\begin{equation}\label{<<X}
\ell \leq \min \{X/c_{1}, \, X/c_{2}\}.
\end{equation}

 We first consider the diagonal $A^{\rm diag}(C,L,\alpha)$ corresponding to the contribution of the terms  satisfying $c_{1}=c_{2}$ 
in \eqref{scis}. The argument in the above remark shows that
there exists an absolute and positive constant $K$ such that 
\begin{equation}\label{148} 
A^{\rm diag} (C,L,\alpha) \ll CL \,\LL^K,
\end{equation}
uniformly for $\alpha$ real, $C$, $L $ and $X\geq 1$.

The off-diagonal part of the sum $A(C,L,\alpha)$ (see \eqref{scis}) is given by
$$
A^{\textnormal{offdiag}}(C,L, \alpha):=\mathop{\sum \sum}_{\substack{c_1, c_2 \sim C\\c_1 \neq c_2 }}\mu(c_1)\mu(c_2)\sum_{\ell\sim L}\lambda_f(\ell c_1)\lambda_f(\ell c_2)
e(\alpha(c_1-c_2)\ell),
$$
where $\ell$ satisfies \eqref{<<X}.
 Let  $\gamma=(c_ 1,c_{2})$.  
We apply  \eqref{H1}   with the choice  $m=c_{1}\ell$, $n=c_{2}\ell$. This gives the equality
 \begin{equation*}
A^{\textnormal{offdiag}}(C,L, \alpha) 
=\sum_{\gamma\leq 2C}\mathop{\sum \sum}_{\substack{c_1, c_2 \sim C\\c_1 \neq c_2\\(c_1,c_2)=\gamma}}\mu(c_1)\mu(c_2)\sum_{\ell\sim L}\sum_{d\mid \ell \gamma}\lambda_f({\ell}^2 c_1 c_2/d^2)
e(\alpha(c_1-c_2)\ell),
\end{equation*}
where $\ell$ satisfies \eqref{<<X}.
Let us further factorize the variables by introducing
\begin{equation*}
c'_{1}=c_{1}\gamma^{-1},
 \text{  and  }c'_{2}=c_{2}\gamma^{-1},
\end{equation*} 
and 
 \begin{equation}\label{DEC1}
 (\gamma, d):=\delta,\  d:=\delta d' \text{ and }\ell:= d' \nu.
 \end{equation}
Note also the equivalences  
$$d\mid \ell \gamma \iff
\frac{d}{(\gamma, d)}\Bigl\vert \,\frac{\gamma}{(\gamma, d)}\cdot \ell \iff \frac{d}{(\gamma, d)}\Bigl\vert \,\ell.
$$
Thus we have,
 \begin{multline}\label{125}
A^{\textnormal{offdiag}}(C,L, \alpha)=\sum_{\gamma}\mu^2 (\gamma)
\underset{\substack{1<c'_{1},\, c'_{2}\sim C\gamma^{-1}\\ (\gamma, c'_{1}c'_{2})=(c'_{1},c'_{2})=1}}{\sum\ \sum}
\mu (c'_{1}c'_{2})\,
\sum_{\delta \mid \gamma}\ \sum_{(d', \gamma\delta^{-1})=1} \\
\sum_{\nu \sim Ld'^{-1}}
\lambda_f \Bigl(\frac{c'_{1}c'_{2}\gamma^2}{\delta^2} \cdot \nu^2\Bigr) e \Bigl( \alpha\, \gamma\, d' (c'_{1}-c'_{2}) \nu 
\Bigr), 
\end{multline}
where $\nu$ satisfies the inequality
\begin{equation}\label{<<<X}
\nu\leq \min \bigl\{X/(\gamma \, c'_{1}\, d'), X/(\gamma\, c'_{2}\, d')
\bigr\}.
\end{equation}
Let $D'=D'(C,L)(<L)$ be a parameter to be fixed later. We split the sum $A^{\rm offdiag }(C,L,\alpha)$ into
 \begin{equation}\label{red}
 A^{\rm offdiag }(C,L,\alpha)= A_{<D'}^{\rm offdiag }(C,L,\alpha)+A_{\geq D'}^{\rm offdiag }(C,L,\alpha),
 \end{equation}
according as $d'<D'$ or $d'\geq D'$ in the sum \eqref{125}. 
  By Lemma \ref{Millgen}, we  obtain the upper bound
 \begin{multline}\label{lam}A_{<D'}^{\rm offdiag }(C,L,\alpha)\ll (CL)^\eps \sum_{\gamma}\mu^2 (\gamma)
\underset{\substack{1<c'_{1},\, c'_{2}\sim C\gamma^{-1}\\ (\gamma, c'_{1}c'_{2})=(c'_{1},c'_{2})=1}}{\sum\ \sum}
\mu^2 (c'_{1}c'_{2})\\
\sum_{\delta \mid \gamma}\ \sum_{\substack{(d', \gamma\delta^{-1})=1\\ d' <D'}} 
{\lambda_f^*}^2 \Bigl(\frac{c'_{1}c'_{2}\gamma^2}{\delta^2}  \Bigr) (L/d')^{\frac{3}{4}+\varepsilon}.
\end{multline} 
By Lemma \ref{propto} \eqref{b)}, we know that ${\lambda_f^*}^2 \bigl(\frac{c'_{1}c'_{2}\gamma^2}{\delta^2}  \bigr)\leq {\lambda_f^*}^2 \bigl({c'_{1}c'_{2}\gamma^2}  \bigr)$.
 Furthermore, each $c\leq 4C^2$ has $O(C^\eps)$ ways  of being written as $c=c'_{1}c'_{2}\gamma^2$, with $c'_{1}$, $c'_{2}$ squarefree and coprime. Using these remarks, we simplify \eqref{lam} into
 \begin{equation}\label{lam1}A_{<D'}^{\rm offdiag }(C,L,\alpha)\ll {D'}^\frac{1}{4}\, L^\frac{3}{4}\,(CL)^\eps \sum_{c\leq 4C^2}d(c)\, {\lambda_f^*}^2 (c).
\end{equation}
  It remains to note the inequality ${\lambda_f^*}^2(m)\leq {\lambda_f^*}^4 (m)$ to 
apply \eqref{4h} to finally deduce the following bound valid for every $\varepsilon >0$.
 \begin{equation}\label{2310} 
A^{\rm offdiag}_{{<D'}}(C,L,\alpha)\ll_{\eps}  C^2{D'}^\frac{1}{4} L^\frac{3}{4} (CL)^\eps,
\end{equation}
uniformly for $C$, $D'$,  $L$ and $X\geq 1$. The above bound is useful when $D'$ is small. When $D'$ is very close to $L$, we recover the trivial bound  
$A^{\rm offdiag} (C,L,\alpha)\ll C^2 L\LL^K$. In that situation we will benefit from the cancellation of additive 
characters in a long sum over the variable ~$d'$.

The goal now is to give an upper bound for
$A^{\rm offdiag}_{{\geq D'}}(C,L,\alpha)$. 
We start from the expressions \eqref{125}  \& \eqref{red} and rewrite   as
\begin{multline}\label{380}
A_{\geq D'}^{\rm offdiag }(C,L,\alpha) = \sum_{\gamma}\mu^2 (\gamma)
\underset{\substack{1<c'_{1},\, c'_{2}\sim C\gamma^{-1}\\ (\gamma, c'_{1}c'_{2})=(c'_{1},c'_{2})=1}}{\sum\ \sum}
\mu (c'_{1}c'_{2})\,
\sum_{\delta \mid \gamma}\ 
\\
\sum_{\nu }
\lambda_f \Bigl(\frac{c'_{1}c'_{2}\gamma^2}{\delta^2} \cdot \nu^2\Bigr) 
\sum_{\substack{(d', \gamma\delta^{-1})=1\\ d'\geq D',\, d'\sim L/\nu}} e \bigl( \alpha\, \gamma\, d' (c'_{1}-c'_{2}) \nu 
\bigr)
\end{multline}
where now $d'$ verifies the extra condition (see \eqref{<X})
\begin{equation}\label{sun<}
d'\leq \min \bigl\{X/(\gamma \, \nu\, c'_{1}), X/(\gamma\, \nu\, c'_{2})
\bigr\}.
\end{equation}
In the expression \eqref{380}, the variable $d'$ is not  smooth completely, because of the coprimality condition $(d',\gamma \delta ^{-1})=1$. Capturing the coprimality condition by the M\" obius function, we write \eqref{380} as
\begin{multline}\label{392}
A_{\geq D'}^{\rm offdiag }(C,L,\alpha) = \sum_{\gamma}\mu^2 (\gamma)
\underset{\substack{1<c'_{1},\, c'_{2}\sim C\gamma^{-1}\\ (\gamma, c'_{1}c'_{2})=(c'_{1},c'_{2})=1}}{\sum\ \sum}
\mu (c'_{1}c'_{2})\,
\sum_{\delta \mid \gamma}\ 
\\
\sum_{\nu }\ 
\lambda_f \Bigl(\frac{c'_{1}c'_{2}\gamma^2}{\delta^2} \cdot \nu^2\Bigr) \sum_{u\mid \gamma\delta ^{-1}} \mu (u)
\sum_{\substack{d'' \geq D'/u\\  d''\sim L/\nu u}} e \bigl( \alpha\, \gamma\, \nu\,(c'_{1}-c'_{2}) u d'' 
\bigr),
\end{multline}
where now \eqref{sun<} is replaced by
\begin{equation}\label{sun<<}
d''\leq \min \bigl\{X/(\gamma \, \nu\, c'_{1}\,  u), X/(\gamma\,\nu\, c'_{2}\,  u)
\bigr\}.
\end{equation}
Taking absolute 
values,  extending the summation over all $u\mid \gamma$ and changing $\delta \mapsto \gamma\delta ^{-1}$, we deduce from \eqref{392}  the inequality
\begin{multline}\label{405}
\bigl\vert\,A_{\geq D'}^{\rm offdiag }(C,L,\alpha) \,\bigr\vert \leq \sum_{\gamma}\mu^2 (\gamma)\,  
\underset{\substack{1<c'_{1},\, c'_{2}\sim C\gamma^{-1}\\ (\gamma, c'_{1}c'_{2})=(c'_{1},c'_{2})=1}}{\sum\ \sum}
\mu^2 (c'_{1}c'_{2})\, \sum_{u\mid \gamma} \mu^2 (u)\,\sum_{\delta \mid \gamma}
\\
\sum_{\nu }\ 
\Bigl\vert\,\lambda_f \bigl( {c'_{1}c'_{2} }{\delta^2}   \nu^2\bigr)\,\Bigr\vert \, \,\Bigl\vert\,
\sum_{\substack{d'' \geq D'/u\\  d''\sim L/\nu u}} e \bigl( \alpha\, \gamma\, \nu\,(c'_{1}-c'_{2}) u d'' 
\bigr)\,\Bigr\vert,
\end{multline}
with the constraint \eqref{sun<<} for the variable $d''$.
We now split the ranges of variations of  the variables $\gamma$, $c'_{1}$, $c'_{2}$,  $u$ and $\nu$
in the right hand side of \eqref{405}
into dyadic segments:
\begin{equation}\label{control10}
\gamma\sim \Gamma,\, c'_{1}\sim C'_{1},\, c'_{2}\sim C'_{2}, \, u\sim U\text{ and } \nu \sim \mathcal{N}.
\end{equation}
We denote by $A(\Gamma, C'_{1}, C'_{2}, U,\mathcal{N})$ the corresponding  contribution.
The number of these subsums is $O (\LL^5)$.
Note that we have
\begin{equation}\label{350}
\Gamma C'_{1}\asymp \Gamma C'_{2}\asymp C, \, U\leq \Gamma, \, L/\mathcal{N}  >D'/2.
\end{equation}
To condense the notations, we define
\begin{equation}\label{defm10}
m:=\gamma\nu u (c'_{1}-c'_{2}).
\end{equation}
Using the well-known bound for sums of additive characters, we have
\begin{equation}\label{3320}
A(\Gamma, C'_{1}, C'_{2}, U, \mathcal{N}) \ll
\sum_{1\leq \vert m\vert \leq M}  g(m) \min \Bigl( \frac{L}{\mathcal{N} U} ,\Vert \alpha m\Vert^{-1}\Bigr),
\end{equation}
 where 
\begin{equation}\label{MMM}
M=16\, \Gamma\, C'_{1} \,U\, \mathcal{N},\  (\asymp  CU\mathcal{N}),
\end{equation}
 $g(m)$  is the weight function
\begin{equation}\label{defg}
g(m):=\sum_{\gamma}  \, 
�\sum_{c'_{1}}\, \sum_{c'_{2}} \, \mu^2 (c'_{1}c'_{2} \gamma)
   \sum_{u\mid \gamma}\, \sum_{\delta \mid \gamma}
 \sum_{\nu}  \bigl\vert \, \lambda_f (c'_{1}c'_{2}\delta ^2 \nu^2)\bigr\vert,
\end{equation}
where   the variables $(\gamma, c'_{1}, c'_{2},   u,\nu)$ also satisfy \eqref{control10} and \eqref{defm10}.
Now we recall the classical lemma (see \cite[p.346]{IK}, for instance).

\begin{lemma}\label{IwKowp.346}  The
inequality    
$$
\sum_{\vert m\vert  \leq M} \min (N, \Vert \alpha m\Vert^{-1} ) \ll (M+N+MNq^{-1}+q)\log 2q
$$
 holds uniformly for $M$ and $N\geq 1$, $\alpha$ real, and any rational number $a/q$ satisfying \eqref{Dir}.
\end{lemma}

To apply Lemma \ref{IwKowp.346} to \eqref{3320}, we first apply the Cauchy-Schwarz inequality with the view to take advantage of the fact that although the size of the coefficients $g(m)$ may be difficult to control, the $\Vert\cdot \Vert_2$--norm of this  sequence can still be estimated by the results of \S 5. 
By Cauchy-Schwarz inequality we obtain
\begin{multline}\label{381}
A(\Gamma, C'_{1}, C'_{2}, U,\mathcal{N})\\
\ll \Bigl( \frac{L}{\mathcal{N} U}\Bigr)^\frac{1}{2}\cdot \Bigl(\sum_{1\leq \vert m\vert \leq M}
g^2 (m) \Bigr)^\frac{1}{2}\cdot \Bigl( M+\frac{L}{\mathcal{N} U}+\frac{LM}{q\mathcal{N} U} +q\Bigr)^\frac{1}{2}\, (\log (2q))^\frac{1}{2}.
\end{multline} 
By Lemma \ref{propto} and the coprimality conditions of the variables $c'_{1}$ and $c'_{2}$,  we get the inequalities
\begin{align*}
\bigl\vert \, \lambda_f (c'_{1}c'_{2}\delta ^2 \nu^2)\bigr\vert &\leq \lambda_f^*(c'_{1}c'_{2})  \lambda_f^*(\delta ^2\nu^2)\\
&\leq \lambda_f^* (c'_{1  })\, \lambda_f^*(c'_{2})\, \lambda_f^* (\gamma^2 \nu^2)\\
&\leq  d(\gamma \nu)\lambda_f^* (c'_{1  })\, \lambda_f^*(c'_{2})\, {\lambda_f^*}^2 (\gamma \nu).
\end{align*}
Inserting this bound into the definition \eqref{defg}, we obtain  the inequality
$$
g(m)\leq \underset{\substack{u,\, \gamma,\, \nu\\ u\gamma\nu\mid m}}{\sum\,\sum\, \sum}\, d(\gamma\,\nu)\,   \, {\lambda_f^*}^2\, (\gamma\, \nu)\,\sum_{c'_{1}}\, \lambda_f^* (c'_{1})\lambda_f^* (c'_{1}+m/(u\gamma\nu)).
$$
By the Cauchy--Schwarz inequality applied to the  sum in $c'_{1}$ (recall that we have $\vert m/(u\gamma \nu)\vert \ll C'_{2}$) and by  \eqref{4h}, we get the upper bound
\begin{align*}
g(m)&\ll C'_{1} \,\LL^K \,\underset{\substack{u,\, \gamma,\, \nu\\ u\gamma\nu\mid m}}{\sum\,\sum\, \sum}\, d(\gamma\,\nu)\,   \, {\lambda_f^*}^2\, (\gamma\, \nu)\\
&\ll {\lambda_f^*}^2(m)\, C'_{1}\, \LL^K \underset{\substack{u,\, \gamma,\, \nu\\ u\gamma\nu\mid m}}{\sum\,\sum\, \sum}\,d(\gamma \nu) \\
&\ll d^5 (m)   {\lambda_f^*}^2 (m) \,C_{1}' \LL^K,
\end{align*}
by using  Lemma \ref{propto} \eqref{b)} and trivial bound on the divisor functions. 
By the above inequality, we have
\begin{align}\label{niss}
\sum_{1\leq \vert m\vert \leq M}g^2 (m)& \ll C_{1}'^2\,\LL^K\, \sum_{1\leq \vert m\vert\leq M} d^{10}(m){\lambda_f^*}^4 (m)\nonumber\\
&\ll C_{1}'^2 M \LL^K\nonumber\\
&\ll C_{1}'^3\,U\, \Gamma \,\mathcal{N}\,\LL^K,
\end{align}
the last lines being consequences of  Proposition \ref{6thmoment}  and the definition \eqref{MMM} of $M$.
 Inserting \eqref{niss} in \eqref{381}, we get the inequality
\begin{equation}\label{488}
A(\Gamma, C'_{1}, C'_{2}, U,\mathcal{N})
\ll {C'_{1}}^\frac{3}{2}\,L^\frac{1}{2}\,\Gamma^\frac{1}{2}\cdot \Bigl( M+\frac{L}{\mathcal{N} U}+\frac{LM}{q\mathcal{N} U} +q\Bigr)^\frac{1}{2}\, \LL^K.
\end{equation}
We must take the supremum of the right hand side of \eqref{488}  under the constraints  \eqref{350} and   \eqref{MMM}. We easily obtain
\begin{align*}
A(\Gamma, C'_{1}, C'_{2}, U,\mathcal{N})
&
\ll 
\bigl( C^3 L\Gamma^{-2}\bigr)^\frac{1}{2}\Bigl( CU\mathcal{N} +L+\frac{CL}{q   } +q\Bigr)^\frac{1}{2}\, \LL^K\\
&\ll 
\bigl( C^3 L\Gamma^{-2}\bigr)^\frac{1}{2}\Bigl( CL D'^{-1}\Gamma+L+\frac{CL}{q   } +q\Bigr)^\frac{1}{2}\, \LL^K.
\end{align*}
By summing over all these subsums we have that if $\alpha$ satisfies \eqref{Dir}, then there exists an absolute constant $K>0$ such that
\begin{equation}\label{4541}
A^{\rm offdiag}_{{\geq D'}}(C,L,\alpha)\ll \bigl( C^2LD'^{-\frac{1}{2}}+ C^\frac{3}{2} L+C^2Lq^{-\frac{1}{2}}+C^\frac{3}{2}L^\frac{1}{2} q^\frac{1}{2} 
\bigr)\LL^K,
\end{equation}
uniformly for $C$, $L\geq 1$ and $1\leq D'\leq L$.

 Recall that we had divided the sum $A(C,L,\alpha)$ into
 \begin{equation}\label{decompo1}
 A(C,L,\alpha) =A^{\rm diag}(C,L,\alpha)+A^{\rm offdiag}_{<D'}(C,L,\alpha) +A^{\rm offdiag}_{\geq D'} (C,L,\alpha).
 \end{equation}
Using
Propositions \eqref{148}, \eqref{2310}  \& \eqref{4541} in \eqref{decompo1} and giving  the value $L^\frac{1}{3}$ to    the parameter $D'$  we complete
  the proof of Theorem \ref{centrall}.
\end{proof}

\subsection{The finishing touches}
By \eqref{type2cauchy}, \eqref{betamoment}, and Theorem \ref{centrall},
we have the inequalities
\begin{align*}
\left|T_2(C, L, \alpha)\right|^2  &\ll L {\LL}^3 A(C, L, \alpha)\\
&\ll_{\eps}  C^2 L^{\frac{11}{6}}(CL)^{\eps} +  \bigl( C^{\frac{3}{2}} L^2 + C^2 L^2 q^{-\frac{1}{2}} + C^{\frac{3}{2}} L^{\frac{3}{2}} q^{\frac{1}{2}} \bigr) {\LL}^K,
\end{align*}
for any $\eps>0$ and for some absolute constant $K>0$.
Therefore, we have the inequality
\begin{equation}\label{final}
T_2(X, \alpha) \ll_{\eps} y^{-\frac{1}{12}} X^{1+\eps }+\bigl(z^{-\frac{1}{4}} X +  q^{-\frac{1}{4}}X  +   q^{\frac{1}{4}}X^\frac{3}{4} \bigr) (\log X)^{K}, 
\end{equation}
by summing over the dyadic segments (see \eqref{condLC}).
Recall that here we are considering only those $\alpha$ for which any rationals $a/q$, $(a, q)=1$ satisfying \eqref{Dirichlet}, also satisfies
$
X/Q < q \leq Q,
$
where $Q = X\exp\left(-\frac{c_1}{3}\sqrt{\log X}\right)$. 
To be precise, we make the choices $
y = z = X^{\frac{1}{5}},
$
and this gives  the upper bound
 $$
T_2(X, \alpha) \ll X\exp\left(-\frac{c_1}{13}\sqrt{\log X}\right),
$$
where the implied constant is absolute. This, together with \eqref{types} and \eqref{type1}, proves Theorem \ref{mainthm1}.


\end{document}